\documentclass[onefignum,onetabnum]{siamart250211}

\usepackage{amsmath,amssymb,mathtools,mathrsfs}
\usepackage{bm}
\usepackage{booktabs}
\usepackage{enumitem}
\usepackage{graphicx}
\usepackage{microtype}
\usepackage{placeins}

\newsiamthm{assumption}{Assumption}
\newsiamremark{remark}{Remark}

\newcommand{\R}{\mathbb R}
\newcommand{\SPD}{\mathbb S_{++}}
\newcommand{\PSD}{\mathbb S_{+}}
\newcommand{\F}{\mathfrak F}
\newcommand{\Ck}{\mathcal C_K}
\newcommand{\Yk}{\mathcal Y_K}
\newcommand{\Kron}{\mathcal K}
\newcommand{\AI}{\mathrm{AI}}
\newcommand{\Len}{\operatorname{Len}}

\newcommand{\diag}{\operatorname{diag}}
\newcommand{\blockdiag}{\operatorname{blockdiag}}
\newcommand{\rank}{\operatorname{rank}}
\newcommand{\tr}{\operatorname{tr}}
\newcommand{\dist}{\operatorname{dist}}

\newcommand{\argminop}{\operatorname*{arg\,min}}

\newcommand{\dd}{\mathrm d}

\newcommand{\fro}[1]{\left\lVert #1\right\rVert_F}
\newcommand{\opnorm}[1]{\left\lVert #1\right\rVert_{\operatorname{op}}}
\newcommand{\norm}[1]{\left\lVert #1\right\rVert}
\newcommand{\ip}[2]{\left\langle #1,#2\right\rangle}
\newcommand{\pospart}[1]{\left[#1\right]_+}

\crefname{assumption}{Assumption}{Assumptions}
\crefname{definition}{Definition}{Definitions}
\crefname{theorem}{Theorem}{Theorems}
\crefname{proposition}{Proposition}{Propositions}
\crefname{lemma}{Lemma}{Lemmas}
\crefname{corollary}{Corollary}{Corollaries}

\headers{Structured Preconditioning in Affine-Invariant Geometry}{Z. Li}

\title{Structured Preconditioning in Affine-Invariant Geometry:
Projection, Certificates, and Kronecker Separation}

\author{Zavier Li\thanks{Xidian University, Xi'an, China
  (\email{zavierli888@gmail.com}).}}

\hypersetup{
  pdftitle={Structured Preconditioning in Affine-Invariant Geometry: Projection, Certificates, and Kronecker Separation},
  pdfauthor={Zavier Li},
  pdfkeywords={positive-definite matrix, affine-invariant metric, Kronecker product, structured preconditioning, condition number, geodesic convexity}
}

\begin{document}
\maketitle

\begin{abstract}
Nearest structured approximation and best structured preconditioning solve
different matrix optimization problems.  We determine their exact relation
for Kronecker positive-definite matrices under the affine-invariant
Riemannian metric.  The Kronecker family is closed and geodesically convex, so
every full matrix has a unique affine-invariant projection.  Its logarithmic
residual satisfies partial-trace normal equations and yields certified point
and objective errors for an Armijo projection solver.  Our central result
shows that this unique projection is also a minimizer of the
Hessian-relative condition number if and only if the extreme spectral states
admit identical tensor marginals.  A computable marginal-mismatch residual
either vanishes at a condition-optimal projection or produces a strict
descent direction.  Two relative spectral levels always force projection
optimality; more strongly, every \(2\times2\) Kronecker projection is
condition-optimal.  An explicit \(2\times3\) construction is therefore a
dimension-minimal strict separation.  Residual-calibrated bounds further bracket
the best attainable Kronecker condition number and the suboptimality of the
projection.  Supporting results place classical diagonal and block Loewner
sandwiches, fixed-basis primal--dual obstructions, and general log-spectral
targets in the same certificate language.  Given validated numerical
enclosures and outward-rounded comparisons, an interval-safe corollary
preserves the soundness of the full
Kronecker tests.  Deterministic small-matrix checks, including a multistart generic
log-factor oracle independent of the partial-trace solver, verify the stated
identities and bounds.
\end{abstract}

\begin{keywords}
positive-definite matrix, affine-invariant metric, Kronecker product,
structured preconditioning, condition number, geodesic convexity
\end{keywords}

\begin{MSCcodes}
15A18, 15B48, 65F08, 65F35, 90C25
\end{MSCcodes}

\section{Introduction}

Positive-definite preconditioners define the geometry in which an optimization
method converts gradients into motion.  A full matrix can reshape every
direction, whereas practical methods often restrict the geometry to diagonal,
block, Kronecker, or low-rank structure.  The restriction creates two matrix
problems.  The first asks whether the family can reach a prescribed
Hessian-relative condition number.  The second asks which structured geometry
is nearest to a given full matrix.  These objectives are related, but they
need not select the same matrix.

Let \(H\in\SPD^d\) and let \(G\in\SPD^d\) be a candidate geometry.  The
preconditioned condition number is
\[
  \kappa(G^{-1}H)
  =
  \kappa(H^{-1/2}GH^{-1/2}).
\]
We measure matrix displacement with the affine-invariant distance
\[
  d_{\AI}(G_0,G_1)
  =
  \fro{\log(G_0^{-1/2}G_1G_0^{-1/2})}.
\]
This metric is congruence invariant and turns the SPD cone into a Hadamard
manifold.  It therefore supplies unique projection onto closed geodesically
convex sets and a natural intrinsic geometry for structured matrix families.

Our first layer treats spectral targets.  For \(H_a\succ0\), let
\[
  \Omega_a(G)
  =
  \phi_a\!\left(
  \log\lambda(G^{-1/2}H_aG^{-1/2})
  \right),
\]
where \(\phi_a\) is convex and permutation invariant.  Coordinatewise
nondecreasing convex aggregation preserves geodesic convexity.  This class
contains log condition number and smooth spectral surrogates.  Its sublevel
sets have unique affine-invariant projection, an explicit exact penalty under
a Slater margin, and a proximal iteration with a global contraction bound.
For closed geodesically convex smooth families, optimality is equivalent to
the whitened spectral subgradient lying in the normal space.

The second layer gives family-specific certificates.  Diagonal and fixed-block
families reach a threshold \(K\) exactly when a structured positive matrix
\(E\) satisfies
\[
  E\preceq H\preceq KE.
\]
This is an LMI feasibility problem and admits explicit dual obstructions.
For Kronecker metrics, the factorization \(G=A\otimes B\) has a scale gauge:
\((cA)\otimes(c^{-1}B)=A\otimes B\).  After quotienting that gauge, we derive
the intrinsic affine-invariant line element and distance, prove that the
Kronecker set is a complete totally geodesic manifold, and characterize its
nearest-point projection by partial-trace normal equations.

The projection geometry exposes the paper's central structural distinction.
The nearest
Kronecker point to \(H\) minimizes \(d_{\AI}(G,H)\), whereas the best
Kronecker preconditioner minimizes \(\kappa(G^{-1}H)\).  We prove an exact
criterion for when the unique projection is a condition-number minimizer, in
terms of the tensor marginals of the extreme eigenspaces of the relative
matrix.  We also turn the criterion into a quantitative marginal-mismatch
residual and a strict descent direction.  Two relative spectral levels force
the criterion.  Independently, every \(2\times2\) Kronecker projection is
condition-optimal, while an explicit \(2\times3\) example gives the
dimension-minimal strict separation.

\paragraph{Contributions.}
The paper makes four main contributions.

\begin{enumerate}[leftmargin=*,itemsep=0.25em]
  \item We characterize exactly when the unique AIRM-nearest Kronecker matrix
  is also condition-optimal, express the criterion through extreme-state
  tensor marginals, derive a computable mismatch/descent certificate, prove
  projection optimality for every two-level relative spectrum and every
  \(2\times2\) problem, and exhibit dimension-minimal strict separation in
  \(2\times3\).

  \item Building on the established Fisher/AIRM quotient geometry of
  Kronecker covariances, we prove global projection normal equations and give
  residual point-error, objective-gap, threshold, and projection-suboptimality
  certificates together with an Armijo solver.

  \item We place these results inside a geodesically convex log-spectral
  target framework with a global normal-response law, exact penalty,
  proximal contraction, and a sharp scale-closed distance lower bound.  The
  full Kronecker certificate admits an interval-safe form when validated
  enclosures of the numerical inputs are supplied.

  \item As supporting comparisons, we restate classical diagonal and block
  reachability as Loewner-sandwich LMIs in the same target-distance notation,
  record dual and fixed-basis obstructions, and validate the matrix identities
  with deterministic, randomized, and independent-oracle checks.
\end{enumerate}

The geometric objects are motivated by diagonal adaptive methods, K-FAC,
Shampoo, and factored statistics, but the exact claims are deterministic
finite-dimensional matrix statements.  Small-SPD experiments verify
projection residuals, reachability, dual obstructions, and strict separation;
they are not used as optimizer-ranking evidence.

\paragraph{Organization.}
\Cref{sec:related} reviews SPD geometry, scaling, and Kronecker approximation.
\Cref{sec:information-rdgc} develops spectral targets and normal response.
\Cref{sec:diag-block} gives diagonal and block certificates.
\Cref{sec:kron-lowrank} develops Kronecker projection and separation.
\Cref{sec:verification,sec:discussion} give verification and scope.

\section{Related Work}
\label{sec:related}

\paragraph{Positive-definite matrix geometry.}
The affine-invariant geometry of positive-definite matrices provides the
distance, exponential map, congruence invariance, and nonpositive curvature
used throughout
\cite{bhatia2007positive,pennec2006riemannian,higham2008functions}.
Geodesically convex optimization and best approximation on SPD manifolds are
developed in
\cite{sra2015conic,lim2004best,bacak2014convex,tumpach2024totally}.
Spectral subdifferential calculus follows the Hermitian convex-analysis
framework of \cite{lewis1996convex}.  Our contribution is the joint
specialization to structured condition targets, explicit penalty and
certificate interfaces, and the exact comparison between metric projection
and condition-number optimization.

\paragraph{Scaling and structured preconditioning.}
Diagonal and block condition-number optimization is connected to optimal
scaling, equilibration, and semidefinite formulations
\cite{bauer1963optimally,vandersluis1969equilibration,
shapiro1982optimally,shapiro1985block,marechal2009optimizing,
lu2011minimizing,jambulapati2020fast,qu2024optimal,gao2023scalable,
ghadimi2025omega}.
In particular, mixed packing--covering SDP solvers and customized large-scale
methods address the computational diagonal problem.  Geodesically convex
preconditioning over symmetric Lie-group quotients provides general
condition-number convexity, projected gradients, and first-order complexity
bounds \cite{dogan2025geodesicpreconditioning}.  We use that established
convexity viewpoint as background.  Applied to a Kronecker transformation
group, that framework supplies condition-objective convexity and first-order
optimization, but does not compare the optimizer with the AIRM projection or
derive the extreme-state marginal criterion.  Our diagonal and block results retain the
exact Loewner reachability question in the notation used later for Kronecker
target distance; the new comparison is between AIRM projection and spectral
condition optimality on the Kronecker SPD submanifold.

Diagonal adaptive methods include AdaGrad and Adam
\cite{duchi2011adagrad,kingma2014adam}; Adafactor and SM3 use factored or
shared statistics \cite{shazeer2018adafactor,anil2019sm3}; K-FAC and KFC use
Kronecker curvature approximations \cite{martens2015kfac,grosse2016kfc};
and Shampoo uses tensor-factor preconditioners
\cite{gupta2018shampoo,anil2020scalable}.  These algorithms motivate the
matrix families.  Identifying a concrete optimizer state with a point in the
family additionally requires a realization theorem specifying damping, bias
correction, update scale, and interpolation.

\paragraph{Kronecker approximation and covariance geometry.}
Kronecker approximation is classical under Frobenius loss
\cite{vanloan1993kronecker}.  Application-driven nearest-Kronecker
preconditioners use Frobenius or operator-specific product approximations for
Toeplitz, imaging, Markov/SAN, and stochastic Galerkin systems
\cite{kamm2000optimal,nagy2006imaging,langville2004san,
langville2004testing,ullmann2010stochastic}.  Those works ask whether a
computationally convenient product approximates a structured operator well or
accelerates a particular solver.  Our comparison instead fixes one Kronecker
family and separates its intrinsic AIRM projection from its exact
Hessian-relative condition-number minimizer.

Separable covariance and matrix-normal models
lead to likelihood and flip-flop procedures
\cite{dutilleul1999mle,werner2008kronecker}.  Geodesic convexity for
covariance estimation and Kronecker structure is established in
\cite{wiesel2012geodesic,wiesel2012kroneckerconvexity}; a broader survey is
\cite{wang2022kroneckersurvey}.  The Fisher/AIRM
quotient geometry, determinant-one factorization, scale gauge, and
dimension-weighted product metric are developed in
\cite{mccormack2025kroneckergeometry}, with related Riemannian algorithms in
\cite{bouchard2021onlinekronecker,simonis2025geodesicvb}.  Work connecting
Shampoo factors to Kronecker approximation studies a complementary algorithmic
question \cite{morwani2024shampoo}.

\begin{center}
\small
\begin{tabular}{p{0.24\linewidth}p{0.31\linewidth}p{0.35\linewidth}}
\toprule
Problem & Established object & Role in this paper \\
\midrule
Kronecker approximation & Frobenius nearest product
\cite{vanloan1993kronecker} & Contrast with intrinsic AIRM projection \\
Separable covariance & Likelihood equations and geodesic convexity
\cite{dutilleul1999mle,wiesel2012kroneckerconvexity} & Distinguish statistical
likelihood from metric projection \\
Kronecker information geometry & Scale quotient and product Fisher/AIRM
metric \cite{mccormack2025kroneckergeometry} & Infrastructure for logarithmic
partial-trace projection residuals \\
Optimal preconditioning & Geodesically convex condition optimization and
first-order methods \cite{dogan2025geodesicpreconditioning} & Exact criterion
for when the AIRM projection belongs to the condition-number argmin \\
\bottomrule
\end{tabular}
\end{center}

Thus the quotient metric itself is supporting infrastructure.  The main
increment here is the global AIRM projection characterization, its residual
error and reachability brackets, the extreme-state marginal criterion for
projection optimality, and strict separation from best Hessian-relative
preconditioning.

\section{Spectral Targets in Affine-Invariant Geometry}
\label{sec:information-rdgc}

We begin with affine-invariant SPD geometry and isolate the static spectral
principles used by the structured families.

\subsection{Relative SPD geometry}

Let
\[
  f(\theta)=\frac12\theta^\top H\theta,
  \qquad H\in\SPD^d.
\]
For a metric \(G\in\SPD^d\), gradient flow is
\(\dot\theta=-G^{-1}H\theta\).  Define the relative metric
\[
  S=H^{-1/2}GH^{-1/2}.
\]
Then \(\kappa(G^{-1}H)=\kappa(S)\), where the left side denotes the ratio of
generalized eigenvalues in \(Hv=\lambda Gv\).

The affine-invariant metric has line element and distance
\[
  \norm{Z}_{S,\AI}
  =\fro{S^{-1/2}ZS^{-1/2}},
  \qquad
  d_{\AI}(S_0,S_1)
  =\fro{\log(S_0^{-1/2}S_1S_0^{-1/2})}.
\]
The SPD cone is a finite-dimensional Hadamard manifold
\cite{bhatia2007positive,pennec2006riemannian}.

For \(K\ge1\), set
\[
  \Ck=\{S\in\SPD^d:\kappa(S)\le K\}.
\]

\begin{theorem}[Full-SPD spectral benchmark]
\label{thm:full-spd-benchmark}
Let \(y_1\le\cdots\le y_d\) be the ordered log-eigenvalues of
\(S_0\in\SPD^d\).  Then
\[
  D_K(S_0)
  :=d_{\AI}(S_0,\Ck)
  =
  \min_{c\in\R}
  \left(
  \sum_{i=1}^d\dist(y_i,[c,c+\log K])^2
  \right)^{1/2}.
\]
\end{theorem}

Thus the unconstrained benchmark clips the relative log-spectrum into an
interval of width \(\log K\).  Restricted families constrain which spectral
and eigenspace motions are admissible.

\subsection{Structured metric families and target distance}

Let \(\F\subset\SPD^d\) be an admissible metric family and define its
Hessian-relative realization
\[
  \mathcal S_\F(H)
  =
  \{H^{-1/2}GH^{-1/2}:G\in\F\}.
\]
A realization specifies a path-connected component, treatment of scale
freedom, an admissible absolutely continuous path class, and the induced or
quotient affine-invariant length.  Smooth embedded, closed geodesically convex,
quotient, and stratified families therefore remain distinct cases.

\begin{definition}[Structured affine-invariant target distance]
\label{def:restricted-complexity}
For \(G_0\in\F\) and \(S_0=H^{-1/2}G_0H^{-1/2}\), define
\[
\begin{aligned}
  D_{K,\F}(S_0;H)
  =
  \inf\{&\Len_{\AI}(S_{[0,T]}):
  S(0)=S_0,\ S(T)\in\Ck,\\
  &S_t\in\mathcal S_\F(H),
  \ S_{[0,T]}\text{ admissible}\}.
\end{aligned}
\]
The value is \(+\infty\) if the relative target is empty or lies in a
different admissible component.
\end{definition}

This target distance records the least intrinsic affine-invariant motion from the initial geometry to a structured condition target.  The endpoint records visible
preconditioning quality, and the admissible curve is eliminated by the
infimum.  For closed geodesically convex realizations it reduces to a unique
target projection; nonconvex or stratified families retain the variational
definition without inheriting that uniqueness.

\subsection{Generating geodesically convex information}

For \(H_a\in\SPD^d\) and a convex permutation-invariant function
\(\phi_a:\R^d\to\R\), define
\[
  \Omega_a(G)
  =
  \phi_a\!\left(
  \log\lambda(G^{-1/2}H_aG^{-1/2})
  \right).
\]

\begin{theorem}[Information-spectrum generation law]
\label{thm:structured-log-spectral-mother}
Let \(\F\subset\SPD^d\) be nonempty, closed, and geodesically convex.  If
\(\rho:\R^q\to\R\) is convex and nondecreasing in every coordinate, then
\[
  \mathcal J(G)
  =
  \rho(\Omega_1(G),\ldots,\Omega_q(G))
\]
is closed and geodesically convex on \(\F\).  The conclusion holds pointwise
for continuously time-varying \(H_{t,a},\phi_{t,a},\rho_t\).

If every \(\phi_a(x+c\mathbf1)=\phi_a(x)\), then \(\mathcal J(cG)=\mathcal
J(G)\).  If \(\phi_a\) is \(L_a\)-Lipschitz and \(\rho\) is
\(L_\rho\)-Lipschitz, replacing \(H_a\) by \(\widehat H_a\) changes the
objective by at most
\[
  \sup_{G\in\F}
  |\widehat{\mathcal J}(G)-\mathcal J(G)|
  \le
  L_\rho
  \left(\sum_{a=1}^qL_a^2
  d_{\AI}(H_a,\widehat H_a)^2\right)^{1/2}.
\]
\end{theorem}

The condition-number objective is obtained from
\[
  \phi_{\mathrm{osc}}(x)=\max_i x_i-\min_i x_i,
  \qquad
  \mathcal J_H(G)=\log\kappa(G^{-1}H).
\]

\subsection{Static action value, exact penalty, and contraction}

The static theorem is stated on an abstract finite-dimensional Hadamard
manifold so that it can be reused by every closed geodesically convex family.

\begin{theorem}[Restricted target projection and exact penalty]
\label{thm:restricted-geometrodynamic-mother}
Let \((\mathcal M,d)\) be a finite-dimensional Hadamard manifold,
\(\F\subset\mathcal M\) a nonempty closed geodesically convex set, and
\(\mathcal J:\F\to\R\cup\{+\infty\}\) proper, closed, geodesically convex,
and bounded below.  Suppose
\[
  \mathcal C_\tau=\{G\in\F:\mathcal J(G)\le\tau\}
\]
is nonempty.  For quadratic kinetic action and terminal indicator
\(\iota_{\mathcal C_\tau}\), the value is
\[
  V(t,G)=\frac{d(G,\mathcal C_\tau)^2}{2(T-t)}.
\]
For \(G_0\in\F\), the least admissible length is
\[
  \mathfrak D_{\tau,\F}^{\mathcal J}(G_0)
  =d(G_0,\mathcal C_\tau)
  =d(G_0,P_\tau),
  \qquad
  P_\tau=P_{\mathcal C_\tau}(G_0).
\]
The projection and constant-speed minimizing geodesic are unique, and for
every \(Y\in\mathcal C_\tau\),
\[
  d(G_0,Y)^2
  \ge
  d(G_0,P_\tau)^2+d(P_\tau,Y)^2.
\]

If families or thresholds are nested, their distances are monotone.  If the
objectives satisfy
\[
  \sup_{G\in\F}|\widehat{\mathcal J}(G)-\mathcal J(G)|\le\varepsilon,
\]
then
\[
  \mathfrak D_{\tau+\varepsilon,\F}^{\mathcal J}(G_0)
  \le
  \mathfrak D_{\tau,\F}^{\widehat{\mathcal J}}(G_0)
  \le
  \mathfrak D_{\tau-\varepsilon,\F}^{\mathcal J}(G_0)
\]
whenever the outer values are finite.

Assume a strict feasible geometry \(G_s\) satisfies
\(\mathcal J(G_s)\le\tau-\sigma\), \(\sigma>0\), and put
\(R=d(G_0,G_s)>0\).  On
\(\mathcal D=\F\cap\overline B_R(G_0)\), every
\[
  \Lambda>\frac{2R^2}{\sigma}
\]
makes
\[
  \Phi_\Lambda(G)
  =
  \frac12d(G_0,G)^2
  +\Lambda\pospart{\mathcal J(G)-\tau}
\]
an exact penalty with unique minimizer \(P_\tau\).  For every \(\eta>0\),
the proximal iteration
\[
  G_{r+1}
  =
  \argminop_{G\in\mathcal D}
  \left\{\Phi_\Lambda(G)+\frac{d(G,G_r)^2}{2\eta}\right\}
\]
satisfies
\[
  d(G_r,P_\tau)^2
  \le
  (1+2\eta)^{-r}d(G_0,P_\tau)^2.
\]
If \(R=0\), then \(G_0=P_\tau\).
\end{theorem}

The proximal recursion is a variational convergence benchmark: its rate is
conditional on solving each global proximal subproblem.  No claim is made that
those subproblems are cheaper than the original target projection for an
arbitrary family.

For the condition-number target, take \(\mathcal J=\mathcal J_H\) and \(\tau=\log K\).  Then
\(\mathfrak D_{\tau,\F}^{\mathcal J}=D_{K,\F}\) for every closed
geodesically convex realization.

\subsection{Normal response and the nearest-versus-best question}

Let \(\F\) now be a nonempty, closed, geodesically convex smooth embedded
submanifold.  In whitened coordinates define
\[
  \widehat T_G\F
  =\{G^{-1/2}UG^{-1/2}:U\in T_G\F\},
  \qquad
  \widehat N_G\F=(\widehat T_G\F)^\perp.
\]
For \(S_a=G^{-1/2}H_aG^{-1/2}\), let
\[
  \mathfrak Z_a(G)
  =
  \partial_X[\phi_a(\lambda(X))]_{X=\log S_a}.
\]
Here \(\partial_X\) is the Euclidean convex subdifferential in the whitened
symmetric-matrix coordinate \(X\); the affine-invariant tangent/cotangent
identification is applied only after this whitening.

\begin{lemma}[Log-spectral pullback identity]
\label{lem:log-spectral-pullback}
Let \(S\in\SPD^d\), let \(\phi:\R^d\to\R\) be convex and
permutation invariant, and set \(F(Y)=\phi(\lambda(Y))\) on
\(\mathbb S^d\).  For every \(X\in\mathbb S^d\),
\[
  \frac{\dd}{\dd t}\bigg|_{t=0+}
  F\!\left(\log(e^{-tX/2}Se^{-tX/2})\right)
  =
  \max_{Z\in\partial F(\log S)}-\tr(ZX).
\]
Consequently, the whitened AIRM subdifferential of the corresponding relative
log-spectral objective is exactly \(-\partial F(\log S)\), including at
repeated eigenvalues.
\end{lemma}

\begin{theorem}[Normal-response law]
\label{thm:normal-response-law}
A point \(G\in\F\) globally minimizes the information objective in
\cref{thm:structured-log-spectral-mother} if and only if there exist
\[
  \theta\in\partial\rho(\Omega_1(G),\ldots,\Omega_q(G)),
  \qquad
  Z_a\in\mathfrak Z_a(G),
\]
such that
\[
  \sum_{a=1}^q\theta_aZ_a\in\widehat N_G\F.
\]
Thus optimality is exact cancellation of the tangent information force.

For a single \(H\), let \(P=P_\F(H)\) be its unique nearest affine-invariant point
in \(\F\) and set \(S=P^{-1/2}HP^{-1/2}\).  Then
\(\log S\in\widehat N_P\F\), and \(P\) also minimizes the log-spectral
difficulty \(\Omega_{\phi,H}\) if and only if
\[
  \mathfrak Z_\phi(P)\cap\widehat N_P\F\ne\emptyset.
\]
\end{theorem}

\begin{corollary}[Sharp scale-closed threshold]
\label{thm:scale-closed-threshold}
Let \(d\ge2\) and let \(\F\subset\SPD^d\) be nonempty and closed under
positive rescaling.  Define
\[
  K_\F^\star(H)=\inf_{G\in\F}\kappa(G^{-1}H),
  \qquad
  \delta_\F(H)=d_{\AI}(H,\F),
  \qquad
  \alpha_d=\sqrt{\frac{d}{\lfloor d^2/4\rfloor}}.
\]
Then
\[
  \log K_\F^\star(H)\ge\alpha_d\delta_\F(H),
  \qquad
  K_\F^\star(H)\ge\exp\{\alpha_d d_{\AI}(H,\F)\}.
\]
The constant \(\alpha_d\) is sharp over general scale-closed families.
\end{corollary}

\begin{proposition}[Monotonicity and full-SPD lower bound]
\label{prop:monotonicity}
If \(\F_1\subset\F_2\) and every admissible \(\F_1\)-path is an admissible
\(\F_2\)-path with the same length, then
\[
  D_{K,\F_1}(S_0;H)\ge D_{K,\F_2}(S_0;H).
\]
In particular, \(D_{K,\F}(S_0;H)\ge D_K(S_0)\) whenever restricted paths are
ambient SPD paths.  Also
\(D_{K_1,\F}\ge D_{K_2,\F}\) for \(K_1\le K_2\).
\end{proposition}

\begin{proposition}[Intrinsic submanifold distance]
\label{prop:submanifold-distance}
If \(\mathcal S_\F(H)\) is a smooth embedded submanifold equipped with the
ambient induced length, then
\[
  D_{K,\F}(S_0;H)
  =
  d_{\mathcal S_\F(H)}
  (S_0,\mathcal S_\F(H)\cap\Ck).
\]
The distance is \(+\infty\) when the target lies outside the admissible
path-connected component.
\end{proposition}

\section{Diagonal and Block Matrix Families}
\label{sec:diag-block}

Diagonal and block families are the closest finite-dimensional models of
coordinatewise and grouped adaptive preconditioning. Their reachability admits a
clean semidefinite characterization. This connects restricted geometric
complexity to the classical literature on scaling and equilibration
\cite{bauer1963optimally,vandersluis1969equilibration,
shapiro1982optimally,shapiro1985block,sinkhorn1967doubly,
knight2013balancing,qu2024optimal}, to condition-number optimization through
convex programming \cite{marechal2009optimizing,lu2011minimizing}, and to
semidefinite feasibility \cite{vandenberghe1996semidefinite,boyd2004convex}.
The role of this section is to place those reachability questions inside the
structured affine-invariant target-distance benchmark and to record primal and
dual certificates in the notation used by \(D_{K,\F}\).

\begin{center}
\begin{tabular}{p{0.28\linewidth}p{0.30\linewidth}p{0.32\linewidth}}
\toprule
Topic & Closest classical object & Role here \\
\midrule
Diagonal/block scaling &
Optimal scaling and equilibration criteria &
Endpoint reachability for a restricted metric family \\
Condition-number minimization &
Variational and convex-programming formulations &
The same target embedded in an affine-invariant path-distance benchmark \\
Semidefinite feasibility &
Primal feasibility and dual separation &
Certificate language for finite \(D_{K,\F}\) or unreachable targets \\
\bottomrule
\end{tabular}
\end{center}

\subsection{Diagonal reachability}

Let
\[
  \F_{\diag}=\{D=\diag(e^{x_1},\ldots,e^{x_d}):x\in\R^d\}.
\]
Inside this family the affine-invariant line element is Euclidean in log
coordinates:
\[
  \Len_{\diag}(x_{[0,T]})=\int_0^T\norm{\dot x_t}_2\,\dd t.
\]
The diagonal target set is
\[
  \mathcal X_K(H)=
  \{x\in\R^d:\kappa(\diag(e^{-x})H)\le K\}.
\]

\begin{lemma}[Diagonal distance identity]
\label{lem:diag-distance-identity}
The map
\[
  x\mapsto H^{-1/2}\diag(e^x)H^{-1/2}
\]
is an isometric realization of the diagonal metric family with its
affine-invariant induced length.  With
\(S_0=H^{-1/2}\diag(e^{x_0})H^{-1/2}\), write
\(D_{K,\diag}(x_0;H):=D_{K,\diag}(S_0;H)\).  The set
\(\mathcal X_K(H)\) is closed, and
\[
  D_{K,\diag}(x_0;H)=\dist_{\ell^2}(x_0,\mathcal X_K(H)).
\]
\end{lemma}

\begin{theorem}[Diagonal LMI reachability]
\label{thm:diag-lmi}
For \(K\ge1\), the following are equivalent:
\begin{enumerate}[leftmargin=*,itemsep=0.1em]
  \item there exists a positive diagonal \(D\) with
  \(\kappa(D^{-1}H)\le K\);
  \item there exists a positive diagonal \(E\) with
  \[
    E\preceq H\preceq K E.
  \]
\end{enumerate}
Consequently,
\[
  K_{\diag}^\ast(H)
  =
  \inf\{K\ge1:\exists E\succ0\ \mathrm{diagonal},\ E\preceq H\preceq KE\}.
\]
For fixed \(K\), diagonal reachability is an LMI feasibility problem.
\end{theorem}

\begin{remark}[Certificate interpretation]
\Cref{thm:diag-lmi} separates reachability from path length. If the LMI is
infeasible, then \(D_{K,\diag}(x_0;H)=+\infty\) for every initial diagonal
metric. If it is feasible, the remaining problem is the intrinsic distance from
the initial log-scale \(x_0\) to the feasible set \(\mathcal X_K(H)\).
Semidefinite duality can therefore provide certificates that no diagonal
preconditioner can reach a desired condition-number threshold.
\end{remark}

\begin{corollary}[Aligned diagonal optimum]
\label{cor:diag-aligned}
If \(H=\diag(h_1,\ldots,h_d)\) and \(G_0\) is diagonal, then
\[
  D_{K,\diag}(S_0;H)=D_K(S_0).
\]
\end{corollary}

Thus diagonal geometry can be optimal when the target curvature is already
aligned with the coordinate axes. Its loss comes from directional mismatch.

\subsection{Block reachability}

Let the coordinates be split as
\[
  \R^d=V_1\oplus\cdots\oplus V_m,
\qquad \dim V_j=d_j,
\]
and let \(\F_{\mathrm{block}}\) be the corresponding block-diagonal positive
definite family.

\begin{theorem}[Block LMI reachability]
\label{thm:block-lmi}
For \(K\ge1\), the following are equivalent:
\begin{enumerate}[leftmargin=*,itemsep=0.1em]
  \item there exists \(G\in\F_{\mathrm{block}}\) with
  \(\kappa(G^{-1}H)\le K\);
  \item there exists \(E\in\F_{\mathrm{block}}\) with
  \[
    E\preceq H\preceq K E.
  \]
\end{enumerate}
Therefore block reachability is also an LMI feasibility problem.
\end{theorem}

\subsection{Dual infeasibility certificates}

The LMI formulation also gives a checkable certificate when a target condition
number is impossible. Let \(\mathcal L\) be either the diagonal symmetric
subspace or a fixed block-diagonal symmetric subspace, and let
\(\Pi_{\mathcal L}\) be the Frobenius-orthogonal projection onto
\(\mathcal L\). For fixed \(K\), the primal reachability problem asks for
\(E\in\mathcal L\) such that
\[
  H-E\succeq0,\qquad KE-H\succeq0.
\]

\begin{proposition}[Structured SDP dual certificate]
\label{prop:structured-dual-certificate}
If there exist \(P,Q\in\PSD^d\) such that
\[
  \Pi_{\mathcal L}(P-KQ)=0,
  \qquad
  \ip{P-Q}{H}<0,
\]
then no \(E\in\mathcal L\) satisfies \(E\preceq H\preceq KE\). Hence the
diagonal or block restricted target is empty at threshold \(K\).
\end{proposition}

\begin{corollary}[Sign-flip rank-one certificate]
\label{cor:sign-flip-certificate}
Let
\[
  T=\blockdiag(s_1 I_{d_1},\ldots,s_m I_{d_m}),
  \qquad s_j\in\{-1,1\},
\]
and let \(\mathcal L\) be the corresponding block-diagonal subspace. If
\[
  \lambda_{\min}(K THT-H)<0,
\]
then the block family cannot reach condition number \(K\). The diagonal case is
the special case \(d_j=1\).
\end{corollary}

The sign-flip certificate is not meant to solve every infeasible SDP. It gives
a cheap, inspectable dual witness for synthetic examples where the obstruction
is cross-block or cross-coordinate coupling.

\begin{proposition}[Block Jacobi upper bound]
\label{prop:block-jacobi}
Let
\[
  G_{\mathrm{BJ}}=\blockdiag(H_{11},\ldots,H_{mm})
\]
with every \(H_{jj}\succ0\), and set
\[
  \widetilde H=G_{\mathrm{BJ}}^{-1/2}HG_{\mathrm{BJ}}^{-1/2}.
\]
Then
\[
  K_{\mathrm{block}}^\ast(H)\le \kappa(\widetilde H).
\]
If \(\opnorm{\widetilde H-I}\le\delta<1\), then
\[
  K_{\mathrm{block}}^\ast(H)\le\frac{1+\delta}{1-\delta}.
\]
\end{proposition}

The block family can absorb all within-block curvature. The residual
difficulty is the normalized off-block interaction.

\section{Kronecker Geometry and Reachability}
\label{sec:kron-lowrank}

Matrix and tensor parameters motivate Kronecker and low-rank restrictions.
These families require more care than diagonal or block families because their
parameterizations include gauge redundancies and nonconvex structure.
They model the structural assumptions behind Kronecker-factored natural
gradient and tensor preconditioning methods \cite{martens2015kfac,gupta2018shampoo}.
They are also related to classical Kronecker approximation and separable
covariance estimation \cite{vanloan1993kronecker,dutilleul1999mle,
werner2008kronecker}; the objective here is affine-invariant metric geometry
rather than Frobenius approximation or statistical likelihood.
This section proves a local quotient line element, a fixed-basis spectral
equivalence, and a general affine-invariant projection theorem for Kronecker
metrics. The projection theorem does not give an elementary closed form in the
fully noncommuting case, but it gives a unique target, partial-trace normal
equations, certified residuals, and auxiliary self-conditioned \(K\)-target
bounds. The Hessian-relative target has an exact endpoint reachability
formulation as a nonconvex Kronecker Loewner sandwich, an associated
expression threshold, residual-calibrated threshold brackets, and an exact
fixed-basis primal--dual obstruction theorem; candidate Kronecker metrics give
direct upper bounds on \(D_{K,\Kron}\) whenever they pass the generalized
condition-number test against \(H\). The intended output is a certifying
interface: a factor-path line element, spectral mismatch signals, fixed-basis
primal or dual witnesses, and noncommuting Kronecker certificates that prove
reachability, prove some global impossibility cases, or abstain.

\subsection{Kronecker spectral surrogate}

Let \(W\in\R^{m\times n}\), so that \(\operatorname{vec}(W)\in\R^{mn}\). A
Kronecker metric has the form
\[
  G=B\otimes A,\qquad A\in\SPD^m,\quad B\in\SPD^n.
\]
If
\[
  A=U\diag(a_i)U^\top,\qquad B=V\diag(b_j)V^\top,
\]
then \(B\otimes A\) has eigenvectors \(v_j\otimes u_i\) and eigenvalues
\(b_j a_i\). Its log-spectrum has the additive form
\[
  \log a_i+\log b_j.
\]
This motivates the additive subspace
\[
  \mathcal A_{\mathrm{kron}}
  =
  \{Z\in\R^{m\times n}:Z_{ij}=\alpha_i+\beta_j\}.
\]

\begin{remark}[Status of the spectral surrogate]
The additive log-spectrum model below is a diagnostic subproblem. It becomes an
intrinsic Kronecker complexity only in the shared fixed-eigenbasis regime of
\cref{prop:kron-fixed-basis}. Outside that regime, Kronecker structure also
constrains eigenvectors; the quotient geometry in
\cref{prop:kron-line-element} and the projection theorem below are the
geometric objects.
\end{remark}

In a fixed shared Kronecker eigenbasis, let \(Y\in\R^{m\times n}\) denote the
current relative log-spectrum. A Kronecker spectral update can move inside the
affine set \(Y+\mathcal A_{\mathrm{kron}}\). Define
\[
  \Yk=\{Z:\max_{i,j}Z_{ij}-\min_{i,j}Z_{ij}\le\log K\}.
\]

\begin{definition}[Kronecker spectral complexity]
The Kronecker spectral surrogate is
\[
  D_{K,\mathrm{kron}}^{\mathrm{spec}}(Y)
  =
  \dist_F\bigl(
  Y,\ (Y+\mathcal A_{\mathrm{kron}})\cap\Yk
  \bigr),
\]
with value \(+\infty\) if the intersection is empty.
\end{definition}

\begin{proposition}[Spectral lower bound]
\label{prop:kron-spec-lower}
Let \(D_K(Y)=\dist_F(Y,\Yk)\). Then
\[
  D_{K,\mathrm{kron}}^{\mathrm{spec}}(Y)\ge D_K(Y).
\]
If the full projection \(\Pi_{\Yk}(Y)\) lies in
\((Y+\mathcal A_{\mathrm{kron}})\), then equality holds.
\end{proposition}

This surrogate is useful but incomplete: the true Kronecker family also
restricts eigenvectors to Kronecker product form.
Consequently, a small value of \(D_{K,\mathrm{kron}}^{\mathrm{spec}}\) is a
certificate only for the shared-eigenbasis subproblem unless accompanied by an
eigenvector-compatibility argument.
For a general Hessian, this missing compatibility can dominate the spectral
width calculation: the additive log-spectrum may be favorable while no nearby
Kronecker eigenbasis aligns with the curvature directions.

\paragraph{Double-centering residual.}
For \(Y\in\R^{m\times n}\), write
\[
  \bar Y=\frac{1}{mn}\sum_{i,j}Y_{ij},\qquad
  r_i=\frac1n\sum_jY_{ij}-\bar Y,\qquad
  c_j=\frac1m\sum_iY_{ij}-\bar Y,
\]
and define the residual
\[
  E_{ij}=Y_{ij}-\bar Y-r_i-c_j.
\]
Then \(E\) is orthogonal to \(\mathcal A_{\mathrm{kron}}\) in Frobenius inner
product. The quantity
\[
  \Delta_{\mathrm{kron}}(Y)=\fro{E}
\]
is invariant under Kronecker spectral updates \(Y\mapsto Y+Z\) with
\(Z\in\mathcal A_{\mathrm{kron}}\). It is therefore a structural mismatch
signal. Turning it into a condition-number lower bound requires additional
width or projection assumptions; the residual alone records nonadditivity.

\subsection{Intrinsic Kronecker quotient geometry}

The factorization has a gauge:
\[
  B\otimes A=(c^{-1}B)\otimes(cA),\qquad c>0.
\]
Thus the intrinsic factor space is a quotient of
\(\SPD^n\times\SPD^m\). For an absolutely continuous factor path define
\[
  X_t=A_t^{-1/2}\dot A_tA_t^{-1/2},\qquad
  Y_t=B_t^{-1/2}\dot B_tB_t^{-1/2}.
\]

\begin{proposition}[Restricted affine-invariant line element]
\label{prop:kron-line-element}
For \(G_t=B_t\otimes A_t\),
\[
  \norm{\dot G_t}_{G_t,\AI}^2
  =
  n\fro{X_t}^2
  +
  m\fro{Y_t}^2
  +
  2\tr(X_t)\tr(Y_t).
\]
The direction \(X_t=\alpha I_m,\ Y_t=-\alpha I_n\) has zero length and is
exactly the factor-scale gauge direction.
\end{proposition}

\begin{proposition}[Distance between Kronecker metrics]
\label{prop:kron-product-distance}
For \(G_0=B_0\otimes A_0\) and \(G_1=B_1\otimes A_1\), set
\[
  L_A=\log(A_0^{-1/2}A_1A_0^{-1/2}),\qquad
  L_B=\log(B_0^{-1/2}B_1B_0^{-1/2}).
\]
Then the ambient affine-invariant distance is
\[
  d_{\AI}(G_0,G_1)^2
  =
  n\fro{L_A}^2+m\fro{L_B}^2
  +2\tr(L_A)\tr(L_B).
\]
\end{proposition}

The line element identifies the degenerate gauge direction and the length
assigned to a specified factor path. The product-distance formula gives exact
distances between two Kronecker metrics. Projecting an arbitrary target Hessian
onto the Kronecker family is a separate problem handled after the fixed-basis
surrogate below.

\begin{proposition}[Fixed-basis exactness of the spectral surrogate]
\label{prop:kron-fixed-basis}
Assume \(H\), \(G_0\), and all feasible \(G_t=B_t\otimes A_t\) share a fixed
Kronecker eigenbasis. Let \(\alpha_i(t)=\log\lambda_i(A_t)\) and
\(\beta_j(t)=\log\lambda_j(B_t)\), and set \(Z_{ij}(t)=\alpha_i(t)+\beta_j(t)\).
Then
\[
  \norm{\dot G_t}_{G_t,\AI}^2=\fro{\dot Z_t}^2.
\]
Consequently, in the fixed-basis log-spectrum subproblem,
\(D_{K,\mathrm{kron}}^{\mathrm{spec}}\) is the intrinsic Kronecker complexity.
\end{proposition}

\subsection{Affine-invariant Kronecker projection}

The fixed-basis surrogate can be strengthened without assuming shared
eigenvectors. Define
\[
  \Kron_{m,n}
  =
  \{B\otimes A:A\in\SPD^m,\ B\in\SPD^n\}
  \subset\SPD^{mn},
\]
and the log-Kronecker subspace
\[
  \mathcal L_{m,n}
  =
  \{I_n\otimes X+Y\otimes I_m:
    X\in\mathbb S^m,\ Y\in\mathbb S^n\}.
\]
For \(L\in\mathbb S^{mn}\), write \(L=[L_{pq}]_{p,q=1}^n\) in
\(m\times m\) blocks and define partial traces
\[
  \operatorname{Tr}_B(L)=\sum_{p=1}^n L_{pp},
  \qquad
  \operatorname{Tr}_A(L)=[\tr(L_{pq})]_{p,q=1}^n.
\]

\begin{theorem}[Log-Kronecker mismatch and closed-form projection]
\label{thm:kron-log-projection}
For \(M\in\SPD^{mn}\),
\[
  M\in\Kron_{m,n}
  \quad\Longleftrightarrow\quad
  \log M\in\mathcal L_{m,n}.
\]
Moreover, the Frobenius projection of \(L\in\mathbb S^{mn}\) onto
\(\mathcal L_{m,n}\) is
\[
  \Pi_{\mathcal L}L
  =
  \tau I_{mn}+I_n\otimes X_0+Y_0\otimes I_m,
\]
where
\[
  \tau=\frac{\tr L}{mn},\qquad
  X_0=\frac1n\operatorname{Tr}_B(L)-\tau I_m,\qquad
  Y_0=\frac1m\operatorname{Tr}_A(L)-\tau I_n.
\]
Thus
\[
  \delta_{\mathrm{kron}}^{\log}(M)
  =
  \fro{\log M-\Pi_{\mathcal L}\log M}
\]
is a closed-form, eigenvector-aware Kronecker mismatch certificate. It
vanishes exactly on \(\Kron_{m,n}\).
\end{theorem}

The double-centering residual above is the special case of
\cref{thm:kron-log-projection} in a fixed Kronecker eigenbasis. Outside that
regime, \(\delta_{\mathrm{kron}}^{\log}\) also detects eigenvector
incompatibility.

\begin{lemma}[Hadamard projection and line-search facts]
\label{lem:hadamard-projection-facts}
Let \(M\) be a finite-dimensional Hadamard manifold and let \(C\subset M\) be
closed and geodesically convex. Every \(x\in M\) has a unique metric
projection \(P_C(x)\). For fixed \(x\), the function
\[
  f(y)=\frac12d(y,x)^2
\]
is \(1\)-strongly geodesically convex on \(C\). If the restricted gradient of
\(f\) is Lipschitz on a compact geodesic neighborhood of an initial sublevel
set, Armijo backtracking along the negative restricted gradient terminates at
every step and accepts step sizes bounded below on that sublevel set.
\end{lemma}

\begin{theorem}[Affine-invariant Kronecker projection]
\label{thm:kron-ai-projection}
The family \(\Kron_{m,n}\) is a closed totally geodesic submanifold of the SPD
cone under \(d_{\AI}\). Hence every \(M\in\SPD^{mn}\) has a unique projection
\[
  G_\star=P_{\Kron}(M)
  =
  \argminop_{G\in\Kron_{m,n}}d_{\AI}(M,G).
\]
If
\[
  R_\star=\log(G_\star^{-1/2}MG_\star^{-1/2}),
\]
then
\[
  \operatorname{Tr}_B(R_\star)=0,\qquad
  \operatorname{Tr}_A(R_\star)=0.
\]
Conversely, any \(G\in\Kron_{m,n}\) satisfying these two equations equals
\(G_\star\). Therefore
\[
  d_{\AI}(M,\Kron_{m,n})=\fro{R_\star}.
\]
\end{theorem}

Theorem \ref{thm:kron-ai-projection} is the intrinsic replacement for the
fixed-basis surrogate. It gives an implicit but global distance computation.
The log-Euclidean certificate remains useful because the
exponential-metric-increasing inequality gives
\[
  d_{\AI}(M,\Kron_{m,n})\ge \delta_{\mathrm{kron}}^{\log}(M).
\]

\subsection{When nearest geometry is condition-optimal}

The unique AIRM projection need not minimize the Hessian-relative condition
number.  The following theorem gives the exact compatibility condition and
the regimes currently known to force or violate it.

\begin{theorem}[Nearest Kronecker geometry versus best preconditioner]
\label{thm:kron-nearest-versus-best}
Let
\[
  P=P_{\Kron}(H),
  \qquad S=P^{-1/2}HP^{-1/2},
\]
and let \(E_+\) and \(E_-\) be the largest- and smallest-eigenvalue
eigenspaces of \(S\). Then \(P\) belongs to
\[
  \argminop_{G\in\Kron_{m,n}}\log\kappa(G^{-1}H)
\]
if and only if there exist density matrices \(Q_+,Q_-\succeq0\) such that
\[
  \tr Q_+=\tr Q_-=1,
  \qquad
  \operatorname{range}(Q_+)\subseteq E_+,
  \qquad
  \operatorname{range}(Q_-)\subseteq E_-,
\]
and
\[
  \operatorname{Tr}_B(Q_--Q_+)=0,
  \qquad
  \operatorname{Tr}_A(Q_--Q_+)=0.
\]
For simple extreme eigenvalues, writing \(u_+\) and \(u_-\) for unit extreme
eigenvectors, this says exactly that the two pure states
\(u_+u_+^\top\) and \(u_-u_-^\top\) have identical marginal density matrices
on both tensor factors.

If \(S\) has exactly two distinct eigenvalues, the condition above always
holds: the nearest affine-invariant Kronecker geometry is automatically a
globally best Kronecker preconditioner.  This universal guarantee does not
extend to unrestricted richer spectra.  In particular, let
\[
  A=\diag(1,-1),
  \qquad
  B=\diag(2,-\tfrac12,-\tfrac32),
  \qquad
  H=\exp(B\otimes A)\in\SPD^6.
\]
Then
\[
  P_{\Kron_{2,3}}(H)=I_6,
  \qquad
  \log\kappa(H)=4,
\]
whereas the Kronecker metric
\[
  G=I_3\otimes\exp(\tfrac14A)
\]
satisfies
\[
  \log\kappa(G^{-1}H)=\frac72.
\]
Thus the unique nearest structured geometry need not be a best structured
preconditioner.
\end{theorem}

\begin{corollary}[Two-by-two completeness and minimal separation dimension]
\label{cor:kron-two-by-two-completeness}
For \(m=n=2\) and every \(H\in\SPD^4\), the unique AIRM projection
\(P_{\Kron}(H)\) is a globally best Kronecker preconditioner.  Consequently,
apart from a trivial one-dimensional factor, strict separation between AIRM
projection and condition-number optimality requires ambient dimension at least
six.  The \(2\times3\) construction in
\cref{thm:kron-nearest-versus-best} is therefore dimension-minimal among real
two-factor SPD Kronecker families.
\end{corollary}

The mechanism is specific to real two-level factors: zero partial traces put
the log residual, after local orthogonal conjugation, in a commuting
Bell-diagonal form, and every Bell spectral projector has maximally mixed
marginals.

\begin{corollary}[Marginal-mismatch descent certificate]
\label{cor:kron-marginal-mismatch}
In the setting of \cref{thm:kron-nearest-versus-best}, define
\[
  \mathcal D_\pm
  =
  \{Q\succeq0:\tr Q=1,\ \operatorname{range}(Q)\subseteq E_\pm\}
\]
and the compact convex set of tangent subgradients
\[
  \mathcal G(P)
  =
  \{\Pi_{\mathcal L}(Q_--Q_+):Q_\pm\in\mathcal D_\pm\}.
\]
Let \(g_\star\) be its unique minimum-Frobenius-norm element and set
\(\mu_{\mathrm{marg}}(P)=\fro{g_\star}\).  Then
\[
  \mu_{\mathrm{marg}}(P)=0
  \quad\Longleftrightarrow\quad
  P\in\argminop_{G\in\Kron_{m,n}}\log\kappa(G^{-1}H).
\]
If \(\mu_{\mathrm{marg}}(P)>0\), the Kronecker geodesic
\[
  G_t=P^{1/2}\exp(-t g_\star)P^{1/2}
\]
satisfies the strict one-sided descent identity
\[
  \frac{\dd}{\dd t}\bigg|_{t=0+}
  \log\kappa(G_t^{-1}H)
  =-\mu_{\mathrm{marg}}(P)^2.
\]
For simple extreme eigenvalues,
\(g_\star=\Pi_{\mathcal L}(u_-u_-^\top-u_+u_+^\top)\).  With multiplicity,
\(g_\star\) is obtained from a finite-dimensional convex conic problem over
the two extreme eigenspaces.
\end{corollary}

\begin{corollary}[Residual certificate for the Kronecker projection solver]
\label{cor:kron-residual-certificate}
For \(G\in\Kron_{m,n}\), define
\[
  R_G(M)=\log(G^{-1/2}MG^{-1/2}),
  \qquad
  V(G)=\Pi_{\mathcal L}R_G(M).
\]
Then
\[
  V(G)=0\quad\Longleftrightarrow\quad G=P_{\Kron}(M),
\]
and
\[
  d_{\AI}(G,P_{\Kron}(M))\le\fro{V(G)},\qquad
  0\le f(G)-f(P_{\Kron}(M))\le\frac12\fro{V(G)}^2,
\]
where \(f(G)=\frac12d_{\AI}(G,M)^2\). Consequently \(\fro{V(G)}\) is a
certified stopping residual for a projected Riemannian solver.
\end{corollary}

\begin{corollary}[Armijo solver for the Kronecker projection]
\label{cor:kron-armijo-solver}
Fix \(G_0\in\Kron_{m,n}\), \(\bar\eta>0\), and
\(\gamma,c\in(0,1)\). Let \(f(G)=\frac12d_{\AI}(G,M)^2\), let
\[
  \mathcal Q_0=\{G\in\Kron_{m,n}:f(G)\le f(G_0)\},
\]
and let \(L_0\) be a Lipschitz constant for the restricted gradient of \(f\) on
a compact geodesic neighborhood of \(\mathcal Q_0\); such an \(L_0\) exists in
the finite-dimensional SPD cone. At iteration \(r\), set
\[
  V_r=\Pi_{\mathcal L}R_{G_r}(M).
\]
If \(V_r=0\), stop. Otherwise choose the largest
\(\eta_r\in\{\bar\eta,\bar\eta\gamma,\bar\eta\gamma^2,\ldots\}\) satisfying
\[
  f\!\left(G_r^{1/2}\exp(\eta_rV_r)G_r^{1/2}\right)
  \le
  f(G_r)-c\eta_r\fro{V_r}^2.
\]
Then
\[
  G_{r+1}=G_r^{1/2}\exp(\eta_rV_r)G_r^{1/2}
\]
is well defined, remains in \(\Kron_{m,n}\), and either terminates at
\(P_{\Kron}(M)\) or converges to \(P_{\Kron}(M)\). Moreover, there exists
\(\underline\eta\in(0,(2c)^{-1}]\), depending only on
\(\mathcal Q_0,L_0,\bar\eta,\gamma\), and \(c\), such that
\[
  f(G_r)-f(P_{\Kron}(M))
  \le
  (1-2c\underline\eta)^r
  \bigl[f(G_0)-f(P_{\Kron}(M))\bigr].
\]
\end{corollary}

\begin{theorem}[Self-conditioned Kronecker \(K\)-target bounds]
\label{thm:kron-k-target}
Let
\[
  \Kron_{m,n,K}=\Kron_{m,n}\cap\Ck.
\]
Here \(\Ck\subset\SPD^{mn}\), so this target constrains the condition number of
the Kronecker metric itself. It is an auxiliary self-conditioning target. For a
fixed Hessian \(H\), the RDGC preconditioning target is instead
\(\{G\in\Kron_{m,n}:\kappa(G^{-1}H)\le K\}\), equivalently
\(\mathcal S_{\Kron}(H)\cap\Ck\) in relative coordinates. The two targets
coincide only under additional compatibility assumptions.

The self-conditioned set \(\Kron_{m,n,K}\) is closed and geodesically convex. If
\(G_\star=P_{\Kron}(M)\), then
\[
  d_{\AI}(M,\Kron_{m,n,K})^2
  \ge
  d_{\AI}(M,\Kron_{m,n})^2
  +
  d_{\AI}(G_\star,\Kron_{m,n,K})^2,
\]
and
\[
  d_{\AI}(M,\Kron_{m,n,K})
  \le
  d_{\AI}(M,\Kron_{m,n})
  +
  d_{\AI}(G_\star,\Kron_{m,n,K}).
\]
Moreover, if \(G_0=B_0\otimes A_0\in\Kron_{m,n}\) and
\[
  A_0=U\diag(e^{a_1},\ldots,e^{a_m})U^\top,\qquad
  B_0=V\diag(e^{b_1},\ldots,e^{b_n})V^\top
\]
with sorted \(a_i\) and \(b_j\), then
\[
  d_{\AI}(G_0,\Kron_{m,n,K})^2
  =
  \min_{x,y}
  \sum_{i=1}^m\sum_{j=1}^n(x_i+y_j-a_i-b_j)^2
\]
subject to
\[
  x_1\le\cdots\le x_m,\qquad
  y_1\le\cdots\le y_n,\qquad
  (x_m-x_1)+(y_n-y_1)\le\log K.
\]
\end{theorem}

\begin{theorem}[Hessian-relative Kronecker reachability]
\label{thm:kron-relative-reachability}
For a fixed \(H\in\SPD^{mn}\), define
\[
  \mathcal R_{K,\Kron}(H)
  =
  \{G\in\Kron_{m,n}:\kappa(G^{-1}H)\le K\}.
\]
Then the following are equivalent:
\begin{enumerate}[leftmargin=*,itemsep=0.2em]
  \item \(\mathcal R_{K,\Kron}(H)\ne\emptyset\);
  \item \(\mathcal C_{K,\Kron}(H)=\mathcal S_{\Kron}(H)\cap\Ck\ne\emptyset\);
  \item there exist \(A\succ0\), \(B\succ0\), and \(\lambda>0\) such that
  \[
    \lambda(B\otimes A)\preceq H\preceq K\lambda(B\otimes A).
  \]
\end{enumerate}
Equivalently, because the scale \(\lambda\) can be absorbed into one factor,
reachability is the feasibility of
\[
  B\otimes A\preceq H\preceq K(B\otimes A),
  \qquad A\succ0,\quad B\succ0.
\]
If these equivalent conditions hold, fix \(G_0\in\Kron_{m,n}\) and set
\(S_0=H^{-1/2}G_0H^{-1/2}\). Then
\[
  D_{K,\Kron}(S_0;H)
  =
  d_{\AI}\bigl(G_0,\mathcal R_{K,\Kron}(H)\bigr).
\]
Moreover, \(\mathcal R_{K,\Kron}(H)\) is a closed geodesically convex subset of
the Kronecker manifold, and the displayed distance is attained at a unique
projection point. Thus the endpoint reachability problem is exact, but the
Kronecker Loewner sandwich is generally a nonconvex feasibility problem in the
ambient matrix variable.

More strongly, suppose a strict feasible certificate \(G_s\in\Kron_{m,n}\)
satisfies
\[
  \log\kappa(G_s^{-1}H)\le\log K-\sigma,
  \qquad \sigma>0,
\]
and let \(R=d_{\AI}(G_0,G_s)>0\). On
\(\mathcal D=\Kron_{m,n}\cap\overline B_R(G_0)\), every
\(\Lambda>2R^2/\sigma\) makes
\[
  \frac12d_{\AI}(G_0,G)^2
  +\Lambda\pospart{\log\kappa(G^{-1}H)-\log K}
\]
an exact penalty for the unique target projection. The proximal iteration of
\cref{thm:restricted-geometrodynamic-mother} converges to that projection with
the global squared-distance contraction factor \((1+2\eta)^{-1}\) per step.
\end{theorem}

\begin{corollary}[Kronecker expression threshold]
\label{cor:kron-kstar}
Define
\[
  K_{\Kron}^\star(H)
  =
  \min_{G\in\Kron_{m,n}}\kappa(G^{-1}H).
\]
The minimum is attained. Moreover, for every \(K\ge1\),
\[
  \mathcal R_{K,\Kron}(H)\ne\emptyset
  \quad\Longleftrightarrow\quad
  K\ge K_{\Kron}^\star(H).
\]
Thus \(K_{\Kron}^\star(H)\) is the intrinsic endpoint threshold for whether a
Kronecker metric can reach the Hessian-relative condition-number target.
\end{corollary}

\begin{theorem}[Residual-calibrated Kronecker threshold bracket]
\label{thm:kron-residual-threshold}
Assume \(mn\ge2\). Let \(H\in\SPD^{mn}\), \(G\in\Kron_{m,n}\), and
\[
  P=P_{\Kron}(H),\qquad
  R_G(H)=\log(G^{-1/2}HG^{-1/2}),\qquad
  V(G)=\Pi_{\mathcal L}R_G(H).
\]
Set
\[
  \eta=\fro{V(G)},\qquad
  \rho=d_{\AI}(H,G),\qquad
  \delta_-=\sqrt{\max\{\rho^2-\eta^2,0\}},
\]
and
\[
  \kappa_G=\kappa(G^{-1}H),\qquad
  \alpha_{mn}=\sqrt{\frac{mn}{\lfloor (mn)^2/4\rfloor}}.
\]
Then
\[
  \delta_-
  \le d_{\AI}(H,\Kron_{m,n})
  \le \rho
\]
and
\[
  \exp(\alpha_{mn}\delta_-)
  \le K_{\Kron}^\star(H)
  \le \kappa_G.
\]
Consequently, if \(\kappa_G\le K\), then \(G\) is a primal
Hessian-relative Kronecker certificate and
\[
  D_{K,\Kron}(S_0;H)\le d_{\AI}(G_0,G)
\]
for every \(G_0\in\Kron_{m,n}\), where \(S_0=H^{-1/2}G_0H^{-1/2}\). Moreover,
with \(\kappa_P=\kappa(P^{-1}H)\),
\[
  e^{-2\eta}\kappa_G\le\kappa_P\le e^{2\eta}\kappa_G.
\]
Writing \(\delta_\star=d_{\AI}(H,\Kron_{m,n})\), the condition-number loss
of using the nearest geometry rather than a best preconditioner obeys
\[
  0
  \le
  \log\frac{\kappa_P}{K_{\Kron}^\star(H)}
  \le
  \log\kappa_P-\alpha_{mn}\delta_\star
  \le
  \log\kappa_G+2\eta-\alpha_{mn}\delta_-.
\]
Thus \(e^{2\eta}\kappa_G\le K\) certifies that the exact projection
\(P_{\Kron}(H)\) reaches the target, while \(e^{-2\eta}\kappa_G>K\) certifies
that this particular projection candidate does not. Finally, if
\[
  \exp(\alpha_{mn}\delta_-)>K,
\]
then the full noncommuting Kronecker target
\(\mathcal R_{K,\Kron}(H)\) is empty.
\end{theorem}

\begin{proposition}[Fixed-basis Hessian-relative Kronecker exactness]
\label{prop:kron-fixed-basis-relative}
Fix orthogonal matrices \(U\in\R^{m\times m}\) and
\(V\in\R^{n\times n}\), and consider the fixed-basis subfamily
\[
  \Kron_{U,V}
  =
  \{(V\diag(e^{b_1},\ldots,e^{b_n})V^\top)
    \otimes
    (U\diag(e^{a_1},\ldots,e^{a_m})U^\top):a\in\R^m,\ b\in\R^n\}.
\]
Assume
\[
  H=(V\otimes U)\diag(h_{ij})(V\otimes U)^\top,
  \qquad h_{ij}>0,
\]
and set \(\ell_{ij}=\log h_{ij}\). Then reachability in
\(\Kron_{U,V}\) at target \(K\) is equivalent to the linear feasibility
problem
\[
  \exists a\in\R^m,\ b\in\R^n,\ c\in\R
  \quad\text{such that}\quad
  c\le \ell_{ij}-a_i-b_j\le c+\log K
  \quad\forall i,j.
\]
Equivalently, the fixed-basis threshold is
\[
  \log K^\star_{U,V}(H)
  =
  \min_{a,b}
  \left[
  \max_{i,j}(\ell_{ij}-a_i-b_j)
  -
  \min_{i,j}(\ell_{ij}-a_i-b_j)
  \right].
\]
Feasibility for this fixed-basis subfamily is a primal certificate for the full
Kronecker family. Infeasibility for the fixed-basis subfamily does not rule out
a noncommuting Kronecker metric with different factor eigenvectors.
\end{proposition}

\begin{theorem}[Fixed-basis primal--dual Kronecker obstruction]
\label{thm:kron-fixed-basis-dual}
In the setting of \cref{prop:kron-fixed-basis-relative}, assume \(m,n\ge2\) and
let \(L=(\ell_{ij})\). The optimal fixed-basis log-threshold
\[
  \tau^\star_{U,V}(H)
  =
  \min_{a,b}
  \left[
  \max_{i,j}(\ell_{ij}-a_i-b_j)
  -
  \min_{i,j}(\ell_{ij}-a_i-b_j)
  \right]
\]
has the exact dual representation
\[
  \tau^\star_{U,V}(H)
  =
  \max_W \ip{W}{L},
\]
where the maximum is over \(W\in\R^{m\times n}\) satisfying
\[
  \sum_j W_{ij}=0\quad\forall i,\qquad
  \sum_i W_{ij}=0\quad\forall j,\qquad
  \sum_{i,j}|W_{ij}|\le2.
\]
Equivalently, any nonzero optimal witness can be normalized so that
\(\sum_{ij}W_{ij}^+=\sum_{ij}W_{ij}^-=1\). Hence any feasible dual witness
with
\[
  \ip{W}{L}>\log K
\]
certifies that the fixed-basis Kronecker target is infeasible.
\end{theorem}

\begin{theorem}[Soundness of \(\mathrm{KRON\text{-}CERTIFY}\)]
\label{thm:kron-certify-soundness}
Consider the following certificate procedure. If a fixed Kronecker eigenbasis
is supplied, solve the primal--dual fixed-basis LP in
\cref{prop:kron-fixed-basis-relative,thm:kron-fixed-basis-dual}.  It records
the fixed-basis feasible status for a primal solution, and the fixed-basis
obstruction status when a feasible dual witness satisfies
\(\ip{W}{L}>\log K\).

In the full noncommuting branch, compute \(\rho,\eta,\delta_-\), and
\(\kappa_G\) for an iterate \(G\in\Kron_{m,n}\) as in
\cref{thm:kron-residual-threshold}.  The four non-inconclusive tests are
\begin{enumerate}[leftmargin=*,itemsep=0.1em]
  \item \(\mathrm{GLOBAL\_INFEASIBLE}\) if
  \(\exp(\alpha_{mn}\delta_-)>K\);
  \item \(\mathrm{FEASIBLE\_DIRECT}\) if \(\kappa_G\le K\);
  \item \(\mathrm{PROJECTED\_FEASIBLE}\) if
  \(e^{2\eta}\kappa_G\le K\);
  \item \(\mathrm{PROJECTED\_INFEASIBLE}\) if
  \(e^{-2\eta}\kappa_G>K\).
\end{enumerate}
If no certificate fires, the procedure records \(\mathrm{INCONCLUSIVE}\).

Every non-inconclusive statement returned by this procedure is correct, with
the stated scope: full-family infeasibility, full-family direct feasibility,
exact-projection feasibility or infeasibility, fixed-basis feasibility, or
fixed-basis obstruction.
\end{theorem}

\begin{corollary}[Interval-safe full Kronecker certification]
\label{cor:kron-certify-interval}
Let the exact quantities in \cref{thm:kron-residual-threshold} have validated
enclosures
\[
  0\le \underline\rho\le\rho,
  \qquad
  0\le\eta\le\overline\eta,
  \qquad
  1\le\underline\kappa\le\kappa_G\le\overline\kappa.
\]
Set
\[
  \underline\delta
  =
  \sqrt{\max\{\underline\rho^2-\overline\eta^2,0\}}.
\]
If the displayed exponentials, products, and final comparisons are also
evaluated with outward rounding, equivalently as validated log-domain
comparisons, then each of the following finite-precision tests is sound:
\[
\begin{array}{rcl}
\exp(\alpha_{mn}\underline\delta)>K
&\Longrightarrow& \mathcal R_{K,\Kron}(H)=\emptyset,\\[0.2em]
\overline\kappa\le K
&\Longrightarrow& G\in\mathcal R_{K,\Kron}(H),\\[0.2em]
e^{2\overline\eta}\overline\kappa\le K
&\Longrightarrow& P_{\Kron}(H)\in\mathcal R_{K,\Kron}(H),\\[0.2em]
e^{-2\overline\eta}\underline\kappa>K
&\Longrightarrow& P_{\Kron}(H)\notin\mathcal R_{K,\Kron}(H).
\end{array}
\]
If none fires, the interval-safe full-family branch returns
\(\mathrm{INCONCLUSIVE}\).  The enclosures may come from interval arithmetic
or from certified backward-error bounds for the eigensolver and matrix
logarithm; unvalidated floating-point estimates alone do not meet the
hypothesis.
\end{corollary}

\begin{proposition}[Hessian-relative Kronecker candidate certificate]
\label{prop:kron-relative-candidate}
Let \(H\in\SPD^{mn}\), \(G_0=B_0\otimes A_0\in\Kron_{m,n}\), and
\(G_c=B_c\otimes A_c\in\Kron_{m,n}\). Define
\[
  S_0=H^{-1/2}G_0H^{-1/2},
  \qquad
  S_c=H^{-1/2}G_cH^{-1/2}.
\]
If
\[
  \kappa(S_c)\le K,
\]
equivalently \(\kappa(G_c^{-1}H)\le K\), then \(G_c\) is a primal
Hessian-relative Kronecker certificate and
\[
  D_{K,\Kron}(S_0;H)\le d_{\AI}(G_0,G_c).
\]
With
\[
  L_A=\log(A_0^{-1/2}A_cA_0^{-1/2}),\qquad
  L_B=\log(B_0^{-1/2}B_cB_0^{-1/2}),
\]
this upper bound has the closed form
\[
  d_{\AI}(G_0,G_c)^2
  =
  n\fro{L_A}^2+m\fro{L_B}^2
  +2\tr(L_A)\tr(L_B).
\]
Thus any candidate Kronecker metric passing the generalized condition-number
test supplies a finite restricted-complexity upper bound. If a proposed
candidate fails the test, no infeasibility conclusion follows.
\end{proposition}

Taking \(G_c=P_{\Kron}(H)\) gives a projection-based sufficient certificate:
one checks the Hessian-relative condition number of the projected Kronecker
metric and, if it is at most \(K\), obtains the path-length upper bound above.
Failure of this projected candidate only means that this particular candidate
does not reach the RDGC target; another Kronecker metric may still do so.

\begin{proposition}[Nonsmooth active \(K\)-target condition]
\label{prop:kron-kkt}
Assume \(K>1\). Let \(Q_\star=P_{\Kron_{m,n,K}}(M)\), and set
\[
  R_\star=\log(Q_\star^{-1/2}MQ_\star^{-1/2}),\qquad
  \omega(Q)=\log\kappa(Q).
\]
If \(\omega(Q_\star)<\log K\), then
\[
  \Pi_{\mathcal L}R_\star=0.
\]
If \(\omega(Q_\star)=\log K\), then there exist \(\mu\ge0\) and
\(W_\star\in\partial\omega(Q_\star)\) such that
\[
  \Pi_{\mathcal L}R_\star=\mu\Pi_{\mathcal L}W_\star.
\]
Here \(\partial\omega(Q_\star)\) denotes the subdifferential in the whitened
AIRM tangent coordinate
\(X\mapsto\omega(Q_\star^{1/2}\exp(X)Q_\star^{1/2})\) at \(X=0\).  It is the
convex hull of
\[
  uu^\top-vv^\top,\qquad
  u\in E_{\max},\ \norm{u}=1,\quad
  v\in E_{\min},\ \norm{v}=1,
\]
with \(E_{\max}\) and \(E_{\min}\) the extremal eigenspaces of \(Q_\star\).
For simple extremal eigenvalues this reduces to the usual multiplier equation
with \(W_\star=u_{\max}u_{\max}^\top-u_{\min}u_{\min}^\top\).
\end{proposition}

\begin{proposition}[Projection diagnostics for K-FAC-style metrics]
\label{prop:kfac-certificates}
Suppose an optimizer uses damped Kronecker metric states
\[
  G_t=(B_t+\lambda_BI_n)\otimes(A_t+\lambda_AI_m).
\]
Let \(H_t\succ0\) be a reference curvature and
\(P_t=P_{\Kron}(H_t)\). Then the quantities
\[
  M_t=d_{\AI}(H_t,P_t),\quad
  A_t=d_{\AI}(P_t,G_t),\quad
  E_t=d_{\AI}(H_t,G_t),\quad
  C_{K,t}=d_{\AI}(P_t,\Kron_{m,n,K})
\]
satisfy
\[
  E_t^2\ge M_t^2+A_t^2,
\]
and
\[
  (M_t^2+C_{K,t}^2)^{1/2}
  \le
  d_{\AI}(H_t,\Kron_{m,n,K})
  \le
  M_t+C_{K,t}.
\]
Here \(A_t\) and the discrete increment \(d_{\AI}(G_t,G_{t+1})\) have the
closed form of \cref{prop:kron-product-distance}, \(M_t\) is computed by
\cref{thm:kron-ai-projection}, and \(C_{K,t}\) is the convex QP in
\cref{thm:kron-k-target}.
These quantities diagnose distance to the Kronecker family, optimizer-state
lag inside that family, and distance to the auxiliary self-conditioned target.
They do not by themselves certify \(\kappa(G_t^{-1}H_t)\le K\). A full
preconditioning certificate must use the Hessian-relative target
\(\{G\in\Kron_{m,n}:\kappa(G^{-1}H_t)\le K\}\) and must specify how the
optimizer state, damping, bias correction, and discrete path realize an
admissible family trajectory.
\end{proposition}

\subsection{Scoped low-rank extension}

Low-rank structures split into three different models:
\begin{enumerate}[leftmargin=*,itemsep=0.2em]
  \item additive metric families, \(G=D+UU^\top\), which are close to practical
  low-rank preconditioners;
  \item log-low-rank families, \(G=\exp(cI+L)\) with \(\rank(L)\le r\), which
  are cleaner for spectral theory;
  \item inverse Woodbury forms, useful for implementation cost.
\end{enumerate}
This paper uses low-rank as an extension layer. The result below is a spectral
surrogate for practical additive forms \(D+UU^\top\); their intrinsic geometry
requires a separate stratified analysis. The clean spectral model assumes a
shared eigenbasis and lets at most \(r\) log-eigenvalue coordinates move away
from an overall shift:
\[
  \mathcal Y_{K,r}(y)
  =
  \{z\in\Yk:\exists c\in\R,\ |\{i:z_i\ne y_i+c\}|\le r\}.
\]
The corresponding spectral low-rank complexity is
\[
  D_{K,r}^{\mathrm{spec}}(y)=\dist_{\ell^2}(y,\mathcal Y_{K,r}(y)).
\]

\begin{proposition}[Low-rank spectral monotonicity]
\label{prop:lowrank-monotonicity}
If \(r_1\le r_2\), then
\[
  D_{K,r_1}^{\mathrm{spec}}(y)
  \ge
  D_{K,r_2}^{\mathrm{spec}}(y)
  \ge
  D_K(y).
\]
When \(r\ge d\), the low-rank spectral complexity equals the full spectral
benchmark \(D_K(y)\).
\end{proposition}

\begin{remark}[Role of low rank]
The proposition captures the intended hierarchy: more spectral correction
directions cannot hurt, and full rank recovers the unconstrained benchmark.
For practical additive forms \(D+UU^\top\), the same monotonicity should be
studied in a stratified metric space rather than in this shared-eigenbasis
spectral surrogate.
The proposition therefore supports the hierarchy of spectral correction
families under the stated shared-eigenbasis surrogate. Claims about a concrete
low-rank preconditioner additionally require its update parameterization,
damping convention, and admissible interpolation to be mapped into one of these
geometric realizations.
\end{remark}

\section{Robustness Interfaces and Certificate Verification}
\label{sec:verification}

\subsection{Loewner proxy robustness}

\begin{proposition}[Proxy inflation]
\label{prop:proxy-inflation}
If
\[
  (1-\varepsilon)H\preceq\widehat H\preceq(1+\varepsilon)H,
  \qquad 0\le\varepsilon<1,
\]
then every \(G\succ0\) satisfies
\[
  \kappa(G^{-1}H)
  \le
  \frac{1+\varepsilon}{1-\varepsilon}
  \kappa(G^{-1}\widehat H).
\]
Consequently, reaching the proxy threshold
\[
  \widehat K=K\frac{1-\varepsilon}{1+\varepsilon}
\]
is sufficient for the true threshold whenever \(\widehat K\ge1\).
\end{proposition}

\begin{proposition}[High-probability proxy interface]
\label{prop:stoch-proxy-upper}
Suppose a random proxy obeys the displayed Loewner event with probability at
least \(1-\delta\).  If a structured candidate \(G\) satisfies
\(\kappa(G^{-1}\widehat H)\le\widehat K\) on that event, then
\(\kappa(G^{-1}H)\le K\) with probability at least \(1-\delta\).
\end{proposition}

The second statement is an interface.  A finite-sample result requires a
concentration theorem for the chosen estimator \cite{tropp2015matrix}.

\subsection{Deterministic and randomized checks}

The accompanying code supplement contains deterministic and randomized
small-SPD checks for diagonal and block certificates, Kronecker spectral
mismatch, Hessian-relative candidates, projection residuals, and strict
separation.  The entry points are
\begin{itemize}[leftmargin=*,itemsep=0em,topsep=0.2em]
  \item \texttt{synthetic\_certificate\_benchmark.py},
  \item \texttt{random\_spd\_certificate\_suite.py},
  \item \texttt{kron\_certify\_demo.py}.
\end{itemize}
Failed identities raise assertions.  The
reference environment is Python 3.11.4 with NumPy 1.26.4, SciPy 1.10.1,
Matplotlib 3.7.1, and SymPy 1.11.1.

In the randomized suite, seeds \(0,\ldots,63\) are used.  The diagonal target
\(K=3\) is reachable in 32 instances, the block-Jacobi target \(K=3\) is
certified in 58, and the fixed-basis Kronecker spectral target \(K=3\) is
reachable in 57.  All 64 constructed Hessian-relative Kronecker positive
controls reach \(K=2\), while the identity constructed negative control fails
that target in all 64 cases.

The positive controls use
\[
  H=G^{1/2}CG^{1/2},
  \qquad
  \kappa(C)\in[1.1,1.8],
\]
with target \(K=2\).  The counts therefore measure theorem-check coverage in
the declared suite, not population success probabilities.

For an independent projection check, the noncommuting \(2\times2\) example is
also optimized by generic BFGS in a five-dimensional gauge-fixed log-factor
chart.  Three deterministic starts agree on the objective to less than
\(2\times10^{-15}\).  The partial-trace Armijo iterate stops at residual
\(8.257\times10^{-6}\); its objective exceeds the independent value by
\(3.227\times10^{-11}\), below the theorem's residual bound
\(3.409\times10^{-11}\), and the two matrices are AIRM distance
\(7.823\times10^{-6}\) apart.  BFGS uses gradient tolerance \(10^{-7}\);
Armijo uses initial step one, shrink factor \(1/2\), parameter \(10^{-4}\),
and residual tolerance \(10^{-4}\).  These calculations test the theorem
through a distinct parameterization and optimizer; they do not establish a
large-scale complexity claim.

The same deterministic script directly checks the marginal criterion in
\cref{thm:kron-nearest-versus-best,cor:kron-marginal-mismatch,%
cor:kron-two-by-two-completeness}.  For the exact \(2\times3\) separation,
the projection normal residual is zero and
\(\mu_{\mathrm{marg}}=0.816496580927726\).  Hence the predicted one-sided
descent slope is \(-\mu_{\mathrm{marg}}^2=-2/3\); a step of \(10^{-6}\)
gives the finite-difference slope \(-0.666666666759852\).  In a deterministic
\(2\times2\) Bell-canonical instance with projection log-condition \(2.2\),
the projection normal residual is zero and the computed marginal mismatch is
\(3.21\times10^{-16}\).  These two checks exercise respectively the strict
separation and the low-dimensional completeness mechanisms.

The deterministic sensitivity scan varies the diagonal coupling and target
\(K\), block coupling, fixed-basis Kronecker threshold, and low-rank spectral
budget.  It contains 86 finite and 49 unreachable diagonal grid entries; the
maximum block primal and dual numerical violations are
\(2.21\times10^{-16}\) and zero.  The low-rank reachable threshold decreases
monotonically from \(121.51\) at rank zero to \(1\) at full rank, as required
by \cref{prop:lowrank-monotonicity}.  These parameter scans expose certificate
branch changes and numerical residuals rather than estimating a data-generating
success probability.

\begin{figure}[htbp]
  \centering
  \includegraphics[width=\linewidth]{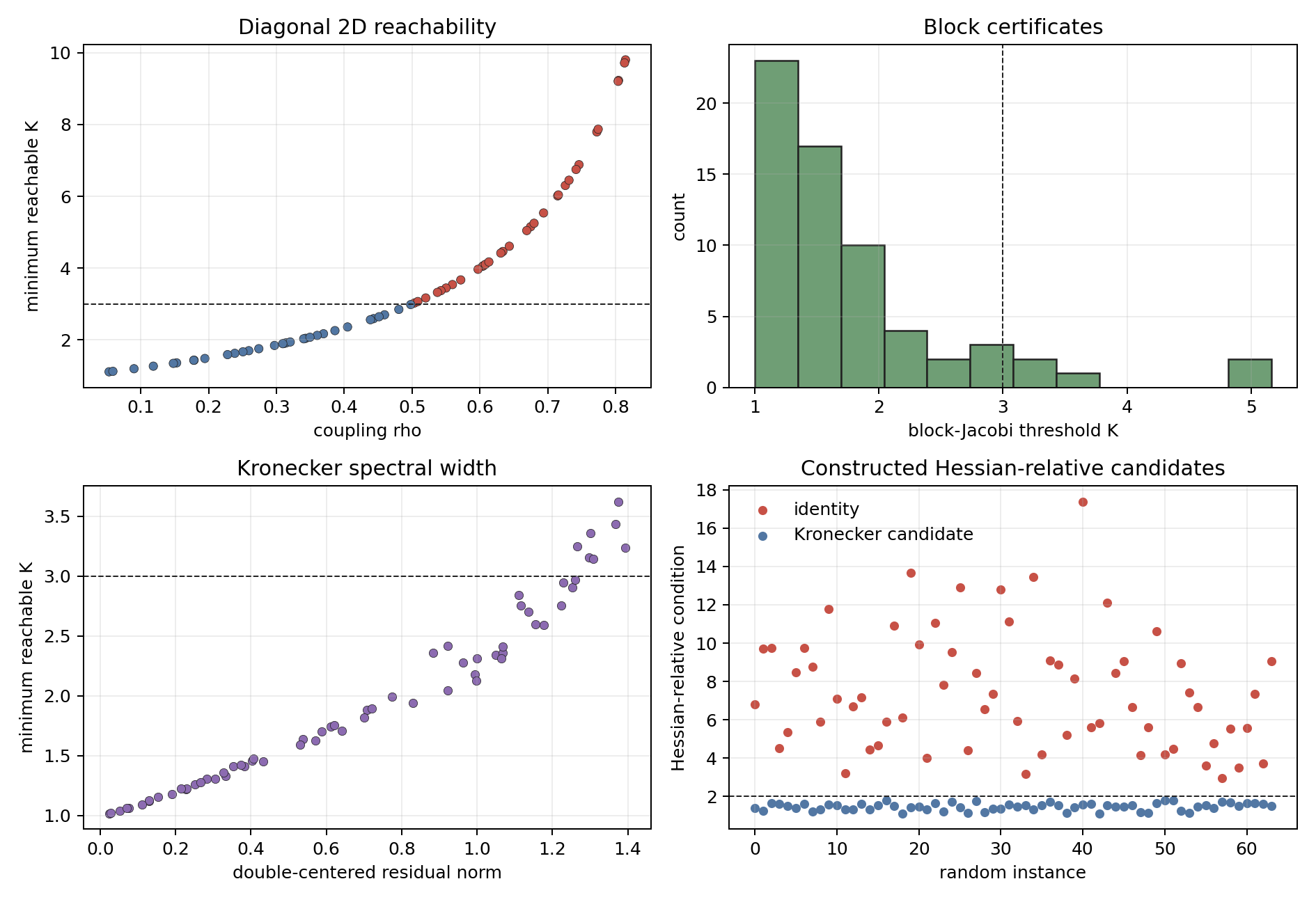}
  \caption{Randomized small-SPD certificate checks over 64 seeds.  The panels
  record diagonal reachability, block-Jacobi thresholds, fixed-basis
  Kronecker spectral width, and constructed Hessian-relative Kronecker
  candidates.  The identity is a constructed negative control.  These are
  theorem checks, not optimizer comparisons.}
  \label{fig:random-spd-suite}
\end{figure}

\FloatBarrier

\section{Discussion}
\label{sec:discussion}

The matrix results separate expressivity, approximation, and
preconditioning quality.  LMI feasibility determines whether a diagonal or
block family can reach a condition target.  Affine-invariant projection
determines the nearest structured geometry.  Condition-number minimization
selects the best preconditioner.  These operations agree in special regimes
and separate in general.

For Kronecker matrices, the factor scale gauge is part of the geometry rather
than a numerical nuisance.  Quotienting the gauge gives a complete totally
geodesic manifold, a closed-form intrinsic line element, and a unique AIRM
projection.  The partial-trace normal equations expose the projection
residual, while the extreme-state marginal condition determines whether that
projection is also condition-optimal.  The marginal-mismatch residual gives a
computable first-order measure of failure.  Every two-level relative spectrum
forces projection optimality.  Every \(2\times2\) Kronecker projection is also
condition-optimal, making the explicit \(2\times3\) separation
dimension-minimal.  Thus the universal guarantee does not extend to
unrestricted richer spectra in the next possible factor dimension.

At the tangent-space level, the same orthogonality calculation extends to a
product of \(k\) SPD factors.  In whitened coordinates its tangent log space
is the sum of the \(k\)
single-mode terms
\[
  I\otimes\cdots\otimes I\otimes X_j\otimes I\otimes\cdots\otimes I.
\]
Repeating the two-factor normal-space proof shows that the projection equations
require every one-mode partial trace of the log residual to vanish, and that
projection optimality is equivalent to choosing admissible density matrices
in the two extreme eigenspaces with matching marginals on every mode.
Fixed-basis dual formulas
and dimension-minimal separation for three or more factors require separate
analysis; the two-factor classification proved here does not presume either.

The exact penalty and proximal contraction provide a global variational
baseline; their rate is conditional on solving each proximal subproblem.
Likewise, the Armijo theorem gives a convergent projection solver but does not
claim that a dense Kronecker projection is inexpensive at neural-network
scale.  Practical implementations require matrix-free partial traces,
structured eigensolvers, and estimator-specific concentration bounds.

The strongest statements are deterministic finite-dimensional matrix
theorems.  Time-dependent curvature, damping, and optimizer states can be
inserted only after a concrete algorithm-to-geometry realization is specified.
The Loewner proxy results state precisely what must be supplied by such a
statistical layer.

The analytic certificates are exact-arithmetic statements.  The interval-safe
form remains rigorous when distance, residual, and condition number are
enclosed by validated numerical bounds and the derived comparisons are
evaluated with outward rounding.  Ordinary floating-point residuals
without such enclosures remain diagnostics; near multiple extreme eigenvalues,
the marginal condition should be solved over the full extreme eigenspaces
rather than from a single numerically selected eigenvector.

In summary, structured preconditioning has three distinct certificates:
reachability of a target, proximity to a full geometry, and
condition-optimality inside the family.  Affine-invariant matrix geometry
places all three in one framework and makes their coincidence and separation
mathematically explicit.

\bibliographystyle{siamplain}
\bibliography{references}

\appendix
\section{Proofs for Spectral Targets and Structured Families}
\label{app:matrix-proofs}

\begin{proof}[Proof of \cref{thm:full-spd-benchmark}]
Let \(S_1\in\Ck\), and let \(z_1\le\cdots\le z_d\) be the ordered
log-eigenvalues of \(S_1\). The affine-invariant distance satisfies the
spectral lower bound
\[
  d_{\AI}(S_0,S_1)\ge \norm{y-z}_2,
\]
where \(y\) is the ordered log-spectrum of \(S_0\). This is the standard
eigenvalue majorization lower bound for the affine-invariant metric on
positive-definite matrices \cite{bhatia2007positive}. Since \(S_1\in\Ck\), the
vector \(z\) has width at most \(\log K\), so \(z_i\in[c,c+\log K]\) for some
\(c\). Hence
\[
  d_{\AI}(S_0,S_1)^2
  \ge
  \sum_{i=1}^d \dist(y_i,[c,c+\log K])^2.
\]
Taking the infimum over \(S_1\in\Ck\) gives the lower bound.

For the reverse inequality, fix \(c\) and define
\[
  z_i(c)=\Pi_{[c,c+\log K]}(y_i),
\]
the scalar projection of \(y_i\) onto the interval. The vector \(z(c)\) has
width at most \(\log K\). Choose \(S_1\) with the same eigenvectors as \(S_0\)
and log-eigenvalues \(z_i(c)\). Then \(S_1\in\Ck\) and the commuting
affine-invariant distance is exactly \(\norm{y-z(c)}_2\). Minimizing over
\(c\in\R\) proves the formula.
\end{proof}

\begin{proof}[Proof of \cref{prop:monotonicity}]
Under the stated compatibility assumption, every admissible \(\F_1\)-path from
the common initial point \(S_0\) to the \(\F_1\)-target is also an admissible
\(\F_2\)-path with the same length. Moreover
\(\mathcal C_{K,\F_1}(H)\subseteq\mathcal C_{K,\F_2}(H)\). Thus the feasible
path set for \(\F_2\) contains the feasible path set for \(\F_1\), and the
infimum cannot increase. Enlarging the target threshold \(K\) similarly
enlarges the feasible endpoint set. The ambient SPD cone contains every
admissible restricted path considered in the lower-bound statement, giving
\(D_{K,\F}\ge D_K\).
\end{proof}

\begin{proof}[Proof of \cref{prop:submanifold-distance}]
Under the smooth embedded submanifold assumption, admissible paths are exactly
absolutely continuous paths in \(\mathcal S_\F(H)\) from \(S_0\) to
\(\mathcal C_{K,\F}(H)\). The induced length is the ambient affine-invariant
length restricted to that submanifold. Hence the infimum is the intrinsic
distance to the target. If no path connects the component of \(S_0\) to the
target, the intrinsic distance is \(+\infty\).
\end{proof}

\begin{proof}[Proof of \cref{thm:scale-closed-threshold}]
We first record the sharp vector inequality used below. For
\(x\in\R^d\), write
\[
  \operatorname{osc}(x)=\max_i x_i-\min_i x_i,\qquad
  \dist(x,\R\mathbf 1)=\min_{s\in\R}\norm{x-s\mathbf 1}_2.
\]
Let \(\tau=\operatorname{osc}(x)\). Translating \(x\) by a scalar multiple of
\(\mathbf 1\) changes neither side, so assume \(\min_i x_i=0\) and
\(\max_i x_i=\tau\). Then \(0\le x_i\le\tau\). The squared distance to
\(\R\mathbf 1\) is \(\sum_i(x_i-\bar x)^2\), which is the variance numerator.
Among vectors in the box \([0,\tau]^d\), this convex function is maximized at a
vertex. If \(k\) coordinates equal \(\tau\) and \(d-k\) equal \(0\), then
\[
  \sum_i(x_i-\bar x)^2=\frac{k(d-k)}{d}\tau^2
  \le\frac{\lfloor d^2/4\rfloor}{d}\tau^2.
\]
Thus
\[
  \operatorname{osc}(x)
  \ge
  \sqrt{\frac{d}{\lfloor d^2/4\rfloor}}
  \dist(x,\R\mathbf 1).
\]
The constant is attained by placing \(\lfloor d/2\rfloor\) coordinates at one
endpoint and the remaining coordinates at the other endpoint.

Now fix \(Q\in\F\) and let
\[
  \lambda_i(Q^{-1/2}HQ^{-1/2})=\exp(\ell_i).
\]
Then
\[
  \log\kappa(Q^{-1}H)=\operatorname{osc}(\ell).
\]
Because \(\F\) is scale closed, \(cQ\in\F\) for every \(c>0\). The eigenvalues
of \(H^{-1/2}(cQ)H^{-1/2}\) are \(\exp(\log c-\ell_i)\), so
\[
  \inf_{c>0}d_{\AI}(H,cQ)
  =
  \dist(\ell,\R\mathbf 1).
\]
The vector inequality gives
\[
  \log\kappa(Q^{-1}H)
  \ge
  \alpha_d\inf_{c>0}d_{\AI}(H,cQ)
  \ge
  \alpha_d d_{\AI}(H,\F).
\]
Taking the infimum over \(Q\in\F\) proves the lower bound. Sharpness follows
from the scale family \(\F=\{cI:c>0\}\) and any equality vector in the sharp
oscillation-distance inequality, used as the log-spectrum of \(H\).
\end{proof}

\begin{proof}[Proof of
\cref{thm:structured-log-spectral-mother,thm:restricted-geometrodynamic-mother}]
We first prove the log-majorization statement that drives the theorem. For
\(A,B\in\SPD^d\), \(t\in[0,1]\), and \(1\le k\le d\), the exterior-power
representation commutes with the weighted geometric mean:
\[
  \bigwedge^k(A\#_tB)
  =
  \left(\bigwedge^kA\right)\#_t
  \left(\bigwedge^kB\right).
\]
For positive-definite \(X,Y\), monotonicity of the geometric mean gives
\[
  \norm{X\#_tY}_{\mathrm{op}}
  \le
  \norm{X}_{\mathrm{op}}^{1-t}
  \norm{Y}_{\mathrm{op}}^t,
\]
because \(X\preceq\norm{X}_{\mathrm{op}}I\) and
\(Y\preceq\norm{Y}_{\mathrm{op}}I\). Applying this inequality to the
exterior powers yields
\[
  \prod_{i=1}^k\lambda_i^\downarrow(A\#_tB)
  \le
  \prod_{i=1}^k
  \lambda_i^\downarrow(A)^{1-t}
  \lambda_i^\downarrow(B)^t.
\]
Equality holds for \(k=d\) because
\(\det(A\#_tB)=\det(A)^{1-t}\det(B)^t\). Taking logarithms proves
\begin{equation}
  \log\lambda(A\#_tB)
  \prec
  (1-t)\log\lambda(A)+t\log\lambda(B),
  \tag{A.1}
\end{equation}
up to the irrelevant common choice of decreasing order.

Fix an index \(a\), abbreviate \(H=H_a\) and \(\phi=\phi_a\), and set
\[
  s_H(G)=\log\lambda(G^{-1/2}HG^{-1/2}).
\]
Now fix
\(G_0,G_1\in\F\) and let \(G_t=G_0\#_tG_1\). Congruence by
\(H^{-1/2}\) is an affine-invariant isometry and preserves geometric means, so
with \(A_i=H^{-1/2}G_iH^{-1/2}\) we have
\[
  H^{-1/2}G_tH^{-1/2}=A_0\#_tA_1.
\]
The eigenvalues of \(G^{-1/2}HG^{-1/2}\) are the reciprocals of those of
\(H^{-1/2}GH^{-1/2}\). Majorization is preserved by multiplication by
\(-1\), and a convex permutation-invariant function is Schur convex.
Therefore (A.1) and ordinary convexity of \(\phi\) give
\[
\begin{aligned}
  \Omega_{\phi,H}(G_t)
  &\le
  \phi\!\left((1-t)s_H(G_0)+t s_H(G_1)\right)\\
  &\le
  (1-t)\Omega_{\phi,H}(G_0)+t\Omega_{\phi,H}(G_1).
\end{aligned}
\]
This proves geodesic convexity of each component \(\Omega_a\). Since
\(\rho\) is coordinatewise nondecreasing,
\[
  \rho(\Omega_1(G_t),\ldots,\Omega_q(G_t))
  \le
  \rho((1-t)\Omega(G_0)+t\Omega(G_1)),
\]
and convexity of \(\rho\) proves geodesic convexity of \(\mathcal J\).
All functions in the statement are finite convex functions on Euclidean
spaces and hence continuous, so \(\mathcal J\) is closed.

If \(\phi_a(x+c\mathbf1)=\phi_a(x)\), then
\[
  \log\lambda((cG)^{-1/2}H_a(cG)^{-1/2})
  =s_{H_a}(G)-(\log c)\mathbf1,
\]
which proves scale invariance. For stability, the standard affine-invariant
spectral inequality gives
\[
  \norm{s_{H_a}(G)-s_{\widehat H_a}(G)}_2
  \le
  d_{\AI}(H_a,\widehat H_a),
\]
because congruence by \(G^{-1/2}\) is an isometry. The Lipschitz bounds for
\(\phi_a\) and \(\rho\), followed by Cauchy--Schwarz, give the displayed
uniform perturbation estimate.

  We now prove the static projection theorem.  For quadratic kinetic action,
  \(U=0\) and
\(\Gamma=\iota_{\mathcal C_\tau}\), every curve over the
remaining horizon \(h=T-t\) satisfies
\[
  \frac12\int_t^T\norm{\dot\gamma}^2\dd s
  \ge\frac{\Len(\gamma)^2}{2h}
  \ge\frac{d(G,\mathcal C_\tau)^2}{2h}.
\]
Equality is attained by the constant-speed geodesic to the metric projection,
which proves the closed form for \(V\).

A sublevel set of a closed geodesically convex function on a closed
geodesically convex family is closed
and geodesically convex. In a finite-dimensional Hadamard manifold it has a
unique metric projection: existence follows by restricting a minimizing
sequence to a closed bounded ball and applying Hopf--Rinow, while uniqueness
follows from strict geodesic convexity of squared distance. The geodesic from
\(G_0\) to the projection stays in \(\F\) and reaches the target with length
equal to the metric distance. Every admissible curve has length at least the
distance between its endpoints, proving the least-action identity and the
uniqueness of the constant-speed minimizer. The displayed Pythagorean
inequality is the standard projection inequality in a Hadamard manifold.
Family and threshold monotonicity follow from inclusion of the corresponding
target and path sets. Finally,
\[
  |\widehat{\mathcal J}-\mathcal J|\le\varepsilon
\]
immediately gives the two target inclusions; applying target monotonicity gives
the perturbation bracket for the dynamic cost.

It remains to prove exactness and convergence of the algorithm. Let
\(P=P_\tau\). Since \(G_s\) is feasible,
\(d(G_0,P)\le R\), so \(P\in\mathcal D\). Closed balls in a Hadamard
manifold are geodesically convex, hence \(\mathcal D\) is closed and
geodesically convex. Write
\[
  f(G)=\frac12d(G_0,G)^2,
  \qquad h(G)=\mathcal J(G)-\tau.
\]
Take any infeasible \(G\in\mathcal D\), set \(a=h(G)>0\), and follow the
geodesic from \(G\) to \(G_s\). At
\[
  t=\frac{a}{a+\sigma}
\]
the point \(Y=G\#_tG_s\) is feasible, because geodesic convexity gives
\(h(Y)\le(1-t)a-t\sigma=0\). Both endpoints lie in the radius-\(R\) ball, so
\[
  d(G,Y)
  =t\,d(G,G_s)
  \le\frac{2Ra}{a+\sigma}
  \le\frac{2Ra}{\sigma}.
\]
On this ball the function \(f\) is \(R\)-Lipschitz, since
\[
  |f(X)-f(Y)|
  \le
  \frac{d(G_0,X)+d(G_0,Y)}{2}d(X,Y)
  \le R d(X,Y).
\]
Consequently
\[
  \Phi_\Lambda(G)
  =f(G)+\Lambda a
  \ge
  f(Y)+\left(\Lambda-\frac{2R^2}{\sigma}\right)a
  >f(Y)\ge f(P)=\Phi_\Lambda(P).
\]
No infeasible point minimizes \(\Phi_\Lambda\). On the feasible set the
penalty vanishes and \(P\) uniquely minimizes \(f\), proving exactness.

The function \(f\) is \(1\)-strongly geodesically convex on a Hadamard
manifold, while \(\pospart{h}=\max\{h,0\}\) is geodesically convex. Thus
\(\Phi_\Lambda\) is \(1\)-strongly geodesically convex on \(\mathcal D\).
Each proximal objective is proper, closed, and
\((1+\eta^{-1})\)-strongly geodesically convex, so it has a unique minimizer
\(G_{r+1}\in\mathcal D\). Strong convexity at the two minimizers
\(P=\argminop_{\mathcal D}\Phi_\Lambda\) and \(G_{r+1}\) gives
\[
\begin{aligned}
  &\Phi_\Lambda(P)+\frac{1}{2\eta}d(G_r,P)^2\\
  &\quad\ge
  \Phi_\Lambda(G_{r+1})
  +\frac{1}{2\eta}d(G_r,G_{r+1})^2
  +\frac{1+\eta^{-1}}{2}d(G_{r+1},P)^2\\
  &\quad\ge
  \Phi_\Lambda(P)
  +\frac12d(G_{r+1},P)^2
  +\frac{1}{2\eta}d(G_r,G_{r+1})^2
  +\frac{1+\eta^{-1}}{2}d(G_{r+1},P)^2.
\end{aligned}
\]
After cancellation and multiplication by \(2\eta\),
\[
  d(G_r,P)^2
  \ge
  (1+2\eta)d(G_{r+1},P)^2+d(G_r,G_{r+1})^2.
\]
Dropping the last term and iterating proves the stated global contraction.
\end{proof}

\begin{proof}[Proof of \cref{lem:log-spectral-pullback}]
Put \(Y=\log S\) and
\[
  S_t=e^{-tX/2}Se^{-tX/2},
  \qquad
  \dot S_0=-\frac12(XS+SX).
\]
The matrix logarithm is Fr\'echet differentiable on the SPD cone, so
\[
  \frac{\dd}{\dd t}\bigg|_{t=0}\log S_t
  =D\log_S[\dot S_0].
\]
For the finite convex spectral function \(F\), convex directional calculus
gives
\[
  F'(Y;D\log_S[\dot S_0])
  =
  \max_{Z\in\partial F(Y)}
  \tr\!\left(ZD\log_S[\dot S_0]\right).
\]
The spectral subdifferential theorem \cite{lewis1996convex} implies that every
\(Z\in\partial F(Y)\) commutes with \(Y\), hence also with \(S=e^Y\).
The Fr\'echet derivative \(D\log_S\) is self-adjoint in the Frobenius inner
product, and commutation gives
\[
  D\log_S[Z]=S^{-1}Z.
\]
Therefore
\[
\begin{aligned}
  \tr\!\left(ZD\log_S[\dot S_0]\right)
  &=\tr\!\left(S^{-1}Z\dot S_0\right)\\
  &=-\frac12\tr\!\left(S^{-1}ZXS+S^{-1}ZSX\right)\\
  &=-\tr(ZX).
\end{aligned}
\]
Substitution into the support-function formula proves the directional identity.
Since a closed convex subdifferential is determined by its support function,
the whitened AIRM subdifferential is exactly \(-\partial F(Y)\), with no
simplicity assumption on the spectrum.
\end{proof}

\begin{proof}[Proof of \cref{thm:normal-response-law}]
Represent a tangent direction at \(G\) in whitened coordinates by
\(X\in\widehat T_G\F\), so the corresponding geodesic is
\(G_t=G^{1/2}\exp(tX)G^{1/2}\).  The generalized eigenvalues of
\((H_a,G_t)\) are the eigenvalues of
\[
  \exp(-tX/2)S_a\exp(-tX/2).
\]
Applying \cref{lem:log-spectral-pullback} gives, at simple or repeated spectra,
\[
  \partial\Omega_a(G)=-\mathfrak Z_a(G)
\]
in whitened AIRM coordinates.

The finite convex composition rule for the coordinatewise nondecreasing
aggregator \(\rho\) now yields
\[
  \partial\mathcal J(G)
  =
  \left\{-\sum_{a=1}^q\theta_aZ_a:
  \theta\in\partial\rho(\Omega(G)),\
  Z_a\in\mathfrak Z_a(G)\right\}
\]
in whitened coordinates.  On the assumed closed geodesically convex smooth
submanifold, the geodesic subgradient inequality shows that an ambient
subgradient orthogonal to the tangent space is globally sufficient.  The
constrained Fermat rule gives the converse.  Since a normal space is a linear
space, the minus sign is immaterial, proving the aggregate normal-response
criterion.

For \(P=P_\F(H)\), the whitened gradient at \(P\) of the squared-distance
objective is
\[
  -\log(P^{-1/2}HP^{-1/2}).
\]
Projection optimality therefore gives
\(\log S\in\widehat N_P\F\). Applying the first part with \(q=1\) proves the
nearest-versus-best coincidence criterion.
\end{proof}

\begin{proof}[Proof of \cref{lem:diag-distance-identity}]
Congruence by \(H^{-1/2}\) is an isometry for \(d_{\AI}\), so it is enough to
compute distances between diagonal matrices. For
\(D(x)=\diag(e^{x_1},\ldots,e^{x_d})\),
\[
  d_{\AI}(D(x),D(y))
  =
  \fro{\log(D(x)^{-1/2}D(y)D(x)^{-1/2})}
  =
  \norm{x-y}_2,
\]
because the matrices commute. The same calculation applied to absolutely
continuous paths gives the Euclidean line element in \(x\)-coordinates. The
target \(\mathcal X_K(H)\) is the preimage of the closed set
\(\Ck\subset\SPD^d\) under the continuous map
\(x\mapsto H^{-1/2}D(x)H^{-1/2}\), hence is closed. Therefore the restricted
complexity is exactly the Euclidean distance from \(x_0\) to this closed target
set, with the usual convention that the distance to the empty set is
\(+\infty\).
\end{proof}

\begin{proof}[Proof of \cref{thm:diag-lmi}]
Suppose \(\kappa(D^{-1}H)\le K\). Let
\[
  \lambda_{\min}=\min_{v\ne0}\frac{v^\top Hv}{v^\top Dv}.
\]
Then all generalized Rayleigh quotients lie in
\([\lambda_{\min},K\lambda_{\min}]\). With \(E=\lambda_{\min}D\),
\[
  E\preceq H\preceq K E.
\]
Conversely, if \(E\preceq H\preceq KE\) for a positive diagonal \(E\), then
every generalized Rayleigh quotient \(v^\top Hv/v^\top Ev\) lies in
\([1,K]\), so \(\kappa(E^{-1}H)\le K\).
\end{proof}

\begin{proof}[Proof of \cref{cor:diag-aligned}]
If \(H\) and \(G_0\) are diagonal, then \(S_0=H^{-1/2}G_0H^{-1/2}\) is diagonal.
The full SPD projection onto \(\Ck\) keeps the eigenvectors of \(S_0\) and
clips only its log-spectrum. Therefore the optimal endpoint and the
affine-invariant geodesic from \(S_0\) to that endpoint remain diagonal. The
diagonal family contains a full SPD shortest path, and \cref{prop:monotonicity}
gives equality.
\end{proof}

\begin{proof}[Proof of \cref{thm:block-lmi}]
The proof is identical to \cref{thm:diag-lmi}, replacing the diagonal cone by
the block-diagonal cone. If \(\kappa(G^{-1}H)\le K\), take
\(E=\lambda_{\min}G\), which is block diagonal. The converse follows from the
generalized Rayleigh quotient bound.
\end{proof}

\begin{proof}[Proof of \cref{prop:structured-dual-certificate}]
Assume, for contradiction, that there exists \(E\in\mathcal L\) with
\[
  H-E\succeq0,\qquad KE-H\succeq0.
\]
For \(P,Q\succeq0\),
\[
  0
  \le
  \ip{P}{H-E}+\ip{Q}{KE-H}
  =
  \ip{P-Q}{H}+\ip{-P+KQ}{E}.
\]
Since \(E\in\mathcal L\) and \(\Pi_{\mathcal L}(P-KQ)=0\), the last term is
zero. Thus any feasible \(E\) would imply \(\ip{P-Q}{H}\ge0\), contradicting
the strict inequality in the certificate.
\end{proof}

\begin{proof}[Proof of \cref{cor:sign-flip-certificate}]
Let \(q\) be a unit eigenvector with
\[
  q^\top(KTHT-H)q<0.
\]
Set
\[
  Q=qq^\top,\qquad P=K(Tq)(Tq)^\top=KTQT.
\]
Then \(P,Q\succeq0\). Because \(T\) is a constant sign on each block,
\(\Pi_{\mathcal L}(TQT-Q)=0\), and hence
\[
  \Pi_{\mathcal L}(P-KQ)=K\Pi_{\mathcal L}(TQT-Q)=0.
\]
Moreover,
\[
  \ip{P-Q}{H}
  =
  q^\top(KTHT-H)q
  <0.
\]
\Cref{prop:structured-dual-certificate} proves infeasibility.
\end{proof}

\begin{proof}[Proof of \cref{prop:block-jacobi}]
The block Jacobi matrix \(G_{\mathrm{BJ}}\) is feasible in the block family, so
the best block condition number is at most the one it achieves. If
\(\opnorm{\widetilde H-I}\le\delta<1\), then
\[
  (1-\delta)I\preceq \widetilde H\preceq (1+\delta)I,
\]
which gives the stated condition-number bound.
\end{proof}

\begin{proof}[Proof of \cref{prop:kron-spec-lower}]
The restricted target
\[
  (Y+\mathcal A_{\mathrm{kron}})\cap\Yk
\]
is a subset of \(\Yk\). Distance to a subset is at least distance to the full
set. If the full projection lies in the restricted target, it attains both
distances.
\end{proof}

\begin{proof}[Proof of \cref{prop:kron-line-element}]
For \(G=B\otimes A\),
\[
  G^{-1/2}\dot G\,G^{-1/2}
  =
  Y\otimes I_m+I_n\otimes X.
\]
Its Frobenius norm squared is
\[
  m\fro{Y}^2+n\fro{X}^2+2\tr(X)\tr(Y),
\]
which is the displayed formula. For \(X=\alpha I_m\) and \(Y=-\alpha I_n\), the
three terms sum to zero.
\end{proof}

\begin{proof}[Proof of \cref{prop:kron-product-distance}]
Using Kronecker identities,
\[
  G_0^{-1/2}G_1G_0^{-1/2}
  =
  (B_0^{-1/2}B_1B_0^{-1/2})
  \otimes
  (A_0^{-1/2}A_1A_0^{-1/2}).
\]
Taking the logarithm gives
\[
  \log(G_0^{-1/2}G_1G_0^{-1/2})
  =
  L_B\otimes I_m+I_n\otimes L_A.
\]
The squared Frobenius norm of this matrix is
\[
  m\fro{L_B}^2+n\fro{L_A}^2+2\tr(L_A)\tr(L_B),
\]
which is the claimed formula.
\end{proof}

\begin{proof}[Proof of \cref{prop:kron-fixed-basis}]
In the fixed Kronecker eigenbasis, the eigenvalues of \(X_t\) and \(Y_t\) are
\(\dot\alpha_i\) and \(\dot\beta_j\). By \cref{prop:kron-line-element},
\[
  \norm{\dot G_t}_{G_t,\AI}^2
  =
  n\sum_i\dot\alpha_i^2+
  m\sum_j\dot\beta_j^2+
  2\left(\sum_i\dot\alpha_i\right)\left(\sum_j\dot\beta_j\right).
\]
On the other hand,
\[
  \fro{\dot Z_t}^2
  =
  \sum_{i,j}(\dot\alpha_i+\dot\beta_j)^2,
\]
which expands to the same expression.
\end{proof}

\begin{proof}[Proof of \cref{thm:kron-log-projection}]
If \(M=B\otimes A\), then the spectral calculus for Kronecker products gives
\[
  \log M=(\log B)\otimes I_m+I_n\otimes(\log A)\in\mathcal L_{m,n}.
\]
Conversely, if
\[
  \log M=Y\otimes I_m+I_n\otimes X,
\]
then the two summands commute, and therefore
\[
  M=\exp(Y\otimes I_m)\exp(I_n\otimes X)=\exp(Y)\otimes\exp(X).
\]
This proves the log characterization.

For the projection formula, decompose
\[
  \mathcal L_{m,n}
  =
  \operatorname{span}\{I_{mn}\}
  \oplus
  \{I_n\otimes X:\tr X=0\}
  \oplus
  \{Y\otimes I_m:\tr Y=0\},
\]
an orthogonal direct sum under the Frobenius inner product. With
\[
  P=\tau I_{mn}+I_n\otimes X_0+Y_0\otimes I_m,
\]
the definitions of \(\tau,X_0,Y_0\) give
\[
  \tr P=\tr L,\qquad
  \operatorname{Tr}_B(P)=\operatorname{Tr}_B(L),\qquad
  \operatorname{Tr}_A(P)=\operatorname{Tr}_A(L).
\]
Thus \(L-P\) is orthogonal to each component of \(\mathcal L_{m,n}\), so
\(P=\Pi_{\mathcal L}L\). The residual statement follows by applying the log
characterization to \(L=\log M\).
\end{proof}

\begin{proof}[Proof of \cref{lem:hadamard-projection-facts}]
Closed geodesically convex subsets of a Hadamard manifold are proximinal and
the projection is unique because squared distance is strictly convex along
geodesics. The same CAT(0) convexity inequality gives, for any geodesic
\(\gamma\) in \(C\),
\[
  f(\gamma_t)
  \le
  (1-t)f(\gamma_0)+tf(\gamma_1)
  -\frac12t(1-t)d(\gamma_0,\gamma_1)^2,
\]
which is \(1\)-strong geodesic convexity of \(f(y)=d(y,x)^2/2\).

For the line-search statement, the descent curve is a geodesic initialized in
the negative restricted-gradient direction. On a compact geodesic neighborhood
where the restricted gradient is \(L\)-Lipschitz, the geodesic descent lemma
gives
\[
  f(\gamma_\eta)
  \le
  f(\gamma_0)-\eta\norm{\operatorname{grad}_C f(\gamma_0)}^2
  +\frac{L}{2}\eta^2\norm{\operatorname{grad}_C f(\gamma_0)}^2.
\]
Thus all sufficiently small positive \(\eta\) satisfy the Armijo inequality.
Compactness makes the admissible upper step size uniform over the sublevel set,
and geometric backtracking then accepts a step bounded below by a positive
constant depending only on the Lipschitz constant and the backtracking
parameters.
\end{proof}

\begin{proof}[Proof of \cref{thm:kron-ai-projection}]
By \cref{thm:kron-log-projection},
\(\Kron_{m,n}=\exp(\mathcal L_{m,n})\). The principal matrix logarithm is a
global diffeomorphism from \(\SPD^{mn}\) to \(\mathbb S^{mn}\), and
\(\mathcal L_{m,n}\) is a linear subspace. Hence \(\Kron_{m,n}\) is a smooth
embedded submanifold. This formulation also removes the factor-scale gauge:
different pairs \((X,Y)\) that differ by
\((X+\alpha I_m,Y-\alpha I_n)\) represent the same element of
\(\mathcal L_{m,n}\).

For \(G_i=B_i\otimes A_i\), the affine-invariant geodesic satisfies
\[
  G_0^{-1/2}G_1G_0^{-1/2}
  =
  (B_0^{-1/2}B_1B_0^{-1/2})
  \otimes
  (A_0^{-1/2}A_1A_0^{-1/2}).
\]
Since \(\exp(t\log(P\otimes Q))=P^t\otimes Q^t\) for \(P,Q\succ0\), the
ambient geodesic between \(G_0\) and \(G_1\) remains in \(\Kron_{m,n}\). Hence
\(\Kron_{m,n}\) is totally geodesic.

It is also closed. If \(B_k\otimes A_k\to G\succ0\), fix the gauge
\(\det A_k=1\). Uniform lower and upper eigenvalue bounds on
\(B_k\otimes A_k\) bound all eigenvalue ratios of \(A_k\); the determinant
gauge then bounds the eigenvalues of \(A_k\) above and below, and the product
eigenvalue bounds do the same for \(B_k\). Passing to a subsequence gives
\(A_k\to A\succ0\) and \(B_k\to B\succ0\), hence \(G=B\otimes A\).

The SPD cone with \(d_{\AI}\) is a Hadamard manifold. By
\cref{lem:hadamard-projection-facts}, projection onto a closed geodesically
convex subset is unique, giving \(G_\star\).
For \(G=B\otimes A\), tangent vectors in whitened coordinates are exactly
\[
  Y\otimes I_m+I_n\otimes X\in\mathcal L_{m,n}.
\]
Indeed, every such \(Z\in\mathcal L_{m,n}\) generates the curve
\(G^{1/2}\exp(tZ)G^{1/2}\in\Kron_{m,n}\), while differentiating any smooth
factor curve \(B_t\otimes A_t\) and whitening by \(G^{-1/2}\) gives an element
of this same subspace. Thus the normal space is the Frobenius orthogonal
complement of \(\mathcal L_{m,n}\) in whitened coordinates.
At the projection point, the logarithmic residual
\[
  R_\star=\log(G_\star^{-1/2}MG_\star^{-1/2})
\]
is orthogonal to every tangent direction. Orthogonality to
\(I_n\otimes X\) and \(Y\otimes I_m\) is exactly
\[
  \operatorname{Tr}_B(R_\star)=0,\qquad
  \operatorname{Tr}_A(R_\star)=0.
\]
Conversely, these equations give the first-order projection condition on a
closed geodesically convex set in a Hadamard manifold, hence the unique
projection. The distance formula is the definition of \(d_{\AI}\).
\end{proof}

\begin{proof}[Proof of \cref{cor:kron-residual-certificate}]
The equivalence \(V(G)=0\Longleftrightarrow G=P_{\Kron}(M)\) is the normal
equation in \cref{thm:kron-ai-projection}. The quantitative bounds use
\cref{lem:hadamard-projection-facts}: the objective
\[
  f(G)=\frac12d_{\AI}(G,M)^2
\]
is \(1\)-strongly geodesically convex on the totally geodesic Hadamard
submanifold \(\Kron_{m,n}\). Let \(G_\star=P_{\Kron}(M)\),
\(v=\operatorname{Log}_G(G_\star)\), and
\(d=\norm{v}=d_{\AI}(G,G_\star)\). Strong convexity along the geodesic from
\(G\) to \(G_\star\) gives
\[
  f(G)-f(G_\star)\le \fro{V(G)}d-\frac12d^2.
\]
Strong convexity at the minimizer gives
\[
  f(G)-f(G_\star)\ge\frac12d^2.
\]
Combining the inequalities yields \(d\le\fro{V(G)}\), and maximizing
\(\fro{V(G)}d-d^2/2\) over \(0\le d\le\fro{V(G)}\) gives the objective-gap
bound.
\end{proof}

\begin{proof}[Proof of \cref{cor:kron-armijo-solver}]
If \(G_r=B_r\otimes A_r\), then \(V_r\in\mathcal L_{m,n}\) has the form
\[
  V_r=Y_r\otimes I_m+I_n\otimes X_r.
\]
The two summands commute, so \(\exp(\eta V_r)=\exp(\eta Y_r)\otimes
\exp(\eta X_r)\), and the trial point remains in \(\Kron_{m,n}\).

Let \(G_\star=P_{\Kron}(M)\). The restricted negative gradient of
\[
  f(G)=\frac12d_{\AI}(G,M)^2
\]
is represented in whitened coordinates by
\(V(G)=\Pi_{\mathcal L}R_G(M)\). The initial sublevel set
\(\mathcal Q_0=\{G\in\Kron_{m,n}:f(G)\le f(G_0)\}\) is closed and bounded in
the complete finite-dimensional Hadamard submanifold \(\Kron_{m,n}\), hence
compact by Hopf--Rinow. Since the SPD cone has no cut locus and the
squared-distance objective is smooth, the restricted gradient is Lipschitz on
a compact geodesic neighborhood of \(\mathcal Q_0\); let \(L_0\) be such a
constant. By \cref{lem:hadamard-projection-facts}, Armijo backtracking
terminates and the accepted step sizes have a positive lower bound
\(\eta_{\min}>0\) depending only on
\(\mathcal Q_0,L_0,\bar\eta,\gamma\), and \(c\)
\cite{absil2008optimization,boumal2023introduction}.

Armijo decrease gives
\[
  f(G_{r+1})\le f(G_r)-c\eta_r\fro{V_r}^2.
\]
Thus \(f(G_r)\) decreases and \(\sum_r\eta_r\fro{V_r}^2<\infty\). Since
\(\eta_r\ge\eta_{\min}\), we have \(\fro{V_r}\to0\). Any cluster point
\(\bar G\) in the compact sublevel set satisfies \(V(\bar G)=0\), hence
\(\bar G=G_\star\) by \cref{cor:kron-residual-certificate}. The cluster point is
unique, so the full sequence converges to \(G_\star\).

For the rate, set
\[
  \underline\eta=\min\{\eta_{\min},(2c)^{-1}\}.
\]
\Cref{cor:kron-residual-certificate} gives
\[
  \fro{V_r}^2\ge2\bigl[f(G_r)-f(G_\star)\bigr].
\]
Combining this with Armijo decrease and \(\eta_r\ge\underline\eta\) yields
\[
  f(G_{r+1})-f(G_\star)
  \le
  (1-2c\underline\eta)\bigl[f(G_r)-f(G_\star)\bigr],
\]
and iteration proves the displayed bound.
\end{proof}

\begin{proof}[Proof of \cref{thm:kron-k-target}]
For the self-conditioned metric target, the condition-number set \(\Ck\) is
closed and geodesically convex because
the weighted matrix geometric mean satisfies
\[
  \lambda_{\max}(S_0\#_tS_1)
  \le
  \lambda_{\max}(S_0)^{1-t}\lambda_{\max}(S_1)^t
\]
and the analogous lower bound for \(\lambda_{\min}\). Together with
\cref{thm:kron-ai-projection}, this makes
\(\Kron_{m,n,K}=\Kron_{m,n}\cap\Ck\) closed and geodesically convex.

The two displayed distance bounds are the Hadamard Pythagorean inequality for
projection onto \(\Kron_{m,n}\), followed by the triangle inequality through
\(G_\star=P_{\Kron}(M)\).

It remains to prove the QP formula. For \(B\otimes A\),
\[
  \kappa(B\otimes A)=\kappa(B)\kappa(A),
\]
so the target condition is
\[
  (x_m-x_1)+(y_n-y_1)\le\log K
\]
for sorted log-eigenvalues \(x_i\) of \(A\) and \(y_j\) of \(B\). By
\cref{prop:kron-product-distance}, the squared distance between Kronecker products
is
\[
  n\,d_{\AI}(A_0,A)^2+m\,d_{\AI}(B_0,B)^2
  +2(\log\det A-\log\det A_0)(\log\det B-\log\det B_0).
\]
The eigenvalue lower bound for the affine-invariant metric and
Hoffman--Wielandt give
\[
  d_{\AI}(A_0,A)^2\ge\sum_i(x_i-a_i)^2,\qquad
  d_{\AI}(B_0,B)^2\ge\sum_j(y_j-b_j)^2.
\]
The determinant terms are the corresponding sums of log-eigenvalue
differences. Expanding gives the lower bound
\[
  \sum_{i,j}(x_i+y_j-a_i-b_j)^2.
\]
Equality is attained by choosing \(A\) and \(B\) with the same eigenvectors as
\(A_0\) and \(B_0\), proving the convex QP.
\end{proof}

\begin{proof}[Proof of \cref{thm:kron-relative-reachability}]
For any \(G\succ0\),
\[
  \kappa(G^{-1}H)
  =
  \kappa(H^{-1/2}GH^{-1/2}),
\]
because \(G^{-1}H\) is similar to
\(H^{1/2}G^{-1}H^{1/2}=(H^{-1/2}GH^{-1/2})^{-1}\), and a positive-definite
matrix and its inverse have the same spectral condition number. Thus
\(G\in\mathcal R_{K,\Kron}(H)\) if and only if
\(H^{-1/2}GH^{-1/2}\in\mathcal C_{K,\Kron}(H)\), proving the equivalence of
the first two items.

For fixed \(G\succ0\), the condition \(\kappa(G^{-1}H)\le K\) is equivalent to
the existence of \(\lambda>0\) such that
\[
  \lambda G\preceq H\preceq K\lambda G.
\]
Indeed, take
\[
  \lambda=\lambda_{\min}(G^{-1/2}HG^{-1/2})
\]
for the forward implication, and use generalized Rayleigh quotients for the
reverse implication. Setting \(G=B\otimes A\) gives the Kronecker Loewner
sandwich. The scalar \(\lambda\) can be absorbed into \(A\) or \(B\), giving
the scale-free form.

It remains to identify the distance. The target can be written as
\[
  \mathcal R_{K,\Kron}(H)
  =
  \Kron_{m,n}\cap H^{1/2}\Ck H^{1/2}.
\]
The set \(\Ck\) is closed and geodesically convex, and congruence by \(H^{1/2}\)
is an affine-invariant isometry, so \(H^{1/2}\Ck H^{1/2}\) is also closed and
geodesically convex. Intersecting with the closed totally geodesic submanifold
\(\Kron_{m,n}\) gives a closed geodesically convex subset of the Kronecker
manifold. The Kronecker manifold is a complete Hadamard submanifold under the
induced affine-invariant metric, so projection onto this nonempty closed
geodesically convex subset exists and is unique.

Finally, congruence by \(H^{-1/2}\) maps paths in
\(\Kron_{m,n}\) isometrically to paths in \(\mathcal S_{\Kron}(H)\). Therefore
the intrinsic distance from \(S_0\) to
\(\mathcal C_{K,\Kron}(H)\) equals the affine-invariant distance from \(G_0\)
to \(\mathcal R_{K,\Kron}(H)\). By \cref{prop:submanifold-distance}, this is
exactly \(D_{K,\Kron}(S_0;H)\).
\end{proof}

\begin{proof}[Proof of \cref{cor:kron-kstar}]
The objective \(\kappa(G^{-1}H)\) is invariant under positive rescaling of
\(G\). Hence a minimizing sequence in \(\Kron_{m,n}\) can be rescaled to satisfy
\(\det G=\det H\). Let
\[
  S=H^{-1/2}GH^{-1/2}.
\]
On this determinant gauge, \(\det S=1\). Since the sequence has bounded
condition number, the eigenvalues of \(S\) are uniformly bounded above and
below. Thus the corresponding sequence of \(S\)'s has a convergent subsequence,
and so does the sequence \(G=H^{1/2}SH^{1/2}\). The family \(\Kron_{m,n}\) is
closed by \cref{thm:kron-ai-projection}, so the limit remains Kronecker and
attains the minimum.

The equivalence
\[
  \mathcal R_{K,\Kron}(H)\ne\emptyset
  \quad\Longleftrightarrow\quad
  K\ge K_{\Kron}^\star(H)
\]
is then immediate from the definition of \(K_{\Kron}^\star(H)\).
\end{proof}

\begin{proof}[Proof of \cref{thm:kron-nearest-versus-best}]
For
\[
  \phi_{\mathrm{osc}}(x)=\max_i x_i-\min_i x_i,
\]
the spectral subdifferential at \(\log S\) is
\[
  \left\{Q_+-Q_-:
  Q_\pm\succeq0,\ \tr Q_\pm=1,\
  \operatorname{range}(Q_+)\subseteq E_+,\
  \operatorname{range}(Q_-)\subseteq E_-\right\}.
\]
By \cref{thm:normal-response-law}, the whitened subgradient for the metric
variable has the opposite sign, \(Q_--Q_+\), and global Kronecker optimality
is equivalent to one such subgradient lying in the normal space of the
Kronecker manifold. The whitened Kronecker tangent space is
\[
  \mathcal L_{m,n}
  =\{Y\otimes I_m+I_n\otimes X:X\in\mathbb S^m,\ Y\in\mathbb S^n\}.
\]
Frobenius orthogonality to \(I_n\otimes X\) for every \(X\) is equivalent to
\(\operatorname{Tr}_B(Q_--Q_+)=0\), and orthogonality to
\(Y\otimes I_m\) for every \(Y\) is equivalent to
\(\operatorname{Tr}_A(Q_--Q_+)=0\). This proves the necessary-and-sufficient
condition. When the extreme eigenvalues are simple, the only density matrices
on the two one-dimensional eigenspaces are \(u_+u_+^\top\) and
\(u_-u_-^\top\), giving the marginal formulation.

Suppose next that \(S\) has exactly two eigenvalues. Write
\[
  \log S=a\Pi_++b\Pi_-,
  \qquad r_\pm=\rank\Pi_\pm.
\]
The nearest-point normal equation gives \(\log S\in\widehat N_P\Kron\).
Because the Kronecker family is scale closed, the whitened scale direction
\(I\) is tangent, so \(\tr\log S=r_+a+r_-b=0\). Since \(a>b\), there is a
constant \(c>0\) such that
\[
  \log S
  =c\left(\frac{\Pi_+}{r_+}-\frac{\Pi_-}{r_-}\right).
\]
Taking \(Q_+=\Pi_+/r_+\) and \(Q_-=\Pi_-/r_-\) makes
\(Q_--Q_+=-c^{-1}\log S\) normal. The first part then proves global
condition optimality.

For the strict counterexample, \(\tr A=\tr B=0\), and hence
\[
  \operatorname{Tr}_B(\log H)
  =\operatorname{Tr}_B(B\otimes A)=0,
  \qquad
  \operatorname{Tr}_A(\log H)=0.
\]
The normal equations in \cref{thm:kron-ai-projection} therefore show that
\(I_6\) is the unique nearest Kronecker point. The relative log-eigenvalues at
the identity are the products of the diagonal entries of \(B\) and \(A\), so
their maximum and minimum are \(2\) and \(-2\), giving
\(\log\kappa(H)=4\).

The proposed metric commutes with \(H\), and
\[
  \log(G^{-1}H)
  =(B-\tfrac14I_3)\otimes A.
\]
The diagonal entries of \(B-\tfrac14I_3\) are
\(7/4,-3/4,-7/4\). Multiplication by the two diagonal entries \(\pm1\) of
\(A\) gives extreme relative log-eigenvalues \(7/4\) and \(-7/4\). Their
oscillation is \(7/2<4\), proving strict noncoincidence.
\end{proof}

\begin{proof}[Proof of \cref{cor:kron-two-by-two-completeness}]
Let \(P=P_{\Kron}(H)\), set \(S=P^{-1/2}HP^{-1/2}\), and write
\(L=\log S\).  The projection normal equations give
\[
  \operatorname{Tr}_A(L)=\operatorname{Tr}_B(L)=0.
\]
Introduce the real matrices
\[
  X=\begin{pmatrix}0&1\\1&0\end{pmatrix},
  \qquad
  Z=\begin{pmatrix}1&0\\0&-1\end{pmatrix},
  \qquad
  J=\begin{pmatrix}0&-1\\1&0\end{pmatrix}.
\]
The symmetric matrices \(I,X,Z\) and the skew-symmetric matrix \(J\) give an
orthogonal tensor basis for real symmetric \(4\times4\) matrices: the symmetric
tensors are generated by the nine products of \(I,X,Z\), together with
\(J\otimes J\).  The two zero-partial-trace conditions remove every term
containing an identity factor.  Hence
\[
  L
  =
  \sum_{r,s\in\{X,Z\}}C_{rs}\,r\otimes s+cJ\otimes J
\]
for a real \(2\times2\) coefficient matrix \(C\) and a scalar \(c\).

Conjugation by \(O\in O(2)\) acts on the span of \(X,Z\) through an
orthogonal \(2\times2\) transformation, and every such transformation is
realized by a suitable \(O\): planar rotations act by twice their angle on
\(\operatorname{span}\{X,Z\}\), and a reflection supplies the other connected
component of \(O(2)\).  Applying the singular value decomposition of
\(C\), choose local orthogonal matrices \(O_A,O_B\) so that conjugation by
\(O_B\otimes O_A\) transforms \(L\) into
\[
  L'=aX\otimes X+bZ\otimes Z+c'J\otimes J.
\]
The three displayed tensor products commute.  Their common orthonormal
eigenbasis is the real Bell basis
\[
\begin{aligned}
  \psi_1&=(e_1\otimes e_1+e_2\otimes e_2)/\sqrt2,&
  \psi_2&=(e_1\otimes e_1-e_2\otimes e_2)/\sqrt2,\\
  \psi_3&=(e_1\otimes e_2+e_2\otimes e_1)/\sqrt2,&
  \psi_4&=(e_1\otimes e_2-e_2\otimes e_1)/\sqrt2.
\end{aligned}
\]
Every Bell projector has both marginal density matrices equal to \(I_2/2\).
If \(\Pi_+\) and \(\Pi_-\) are the extreme spectral projectors of \(L'\),
then
\[
  Q_+=\frac{\Pi_+}{\rank\Pi_+},
  \qquad
  Q_-=\frac{\Pi_-}{\rank\Pi_-}
\]
therefore have identical marginals on both factors.  Local orthogonal
conjugation preserves this equality, so the same is true for the extreme
eigenspaces of \(L\).  \Cref{thm:kron-nearest-versus-best} proves that \(P\)
is condition-optimal.

If one factor has dimension one, the Kronecker family is the full SPD cone in
the other factor and separation is impossible.  The only remaining nontrivial
factor pair below ambient dimension six is \(2\times2\), just handled.  The
existing \(2\times3\) example attains dimension six, proving minimality.
\end{proof}

\begin{proof}[Proof of \cref{cor:kron-marginal-mismatch}]
The whitened metric-variable subdifferential of
\(G\mapsto\log\kappa(G^{-1}H)\) at \(P\) is
\[
  \mathcal W(P)
  =
  \{Q_--Q_+:Q_\pm\in\mathcal D_\pm\}.
\]
Restricting the objective to the Kronecker manifold projects this set onto the
whitened tangent space \(\mathcal L_{m,n}\), giving exactly
\(\mathcal G(P)=\Pi_{\mathcal L}\mathcal W(P)\).  The sets
\(\mathcal D_\pm\) are compact and convex, so \(\mathcal G(P)\) is compact
and convex.  Strict convexity of the squared Frobenius norm gives its unique
minimum-norm element \(g_\star\).

By the constrained Fermat rule, \(P\) is condition-optimal if and only if
\(0\in\mathcal G(P)\), equivalently \(\mu_{\mathrm{marg}}(P)=0\).  Suppose
now that \(g_\star\ne0\).  The variational characterization of the Euclidean
projection of zero onto \(\mathcal G(P)\) gives
\[
  \ip{g-g_\star}{g_\star}\ge0
  \qquad\text{for every }g\in\mathcal G(P).
\]
Hence
\[
  \min_{g\in\mathcal G(P)}\ip{g}{g_\star}
  =\fro{g_\star}^2.
\]
The one-sided directional derivative of a convex function is the support
function of its subdifferential.  Along the whitened tangent direction
\(-g_\star\), it is therefore
\[
  \max_{g\in\mathcal G(P)}\ip{g}{-g_\star}
  =-\fro{g_\star}^2.
\]
Because \(g_\star\in\mathcal L_{m,n}\), the corresponding AIRM geodesic
remains in the Kronecker manifold, proving the descent identity.

For simple extremes, both \(\mathcal D_\pm\) are singletons, which gives the
displayed formula.  With multiplicity, minimizing the norm of a linear image
of \((Q_+,Q_-)\) subject to positive-semidefinite, trace, and support
constraints is a finite-dimensional convex conic problem.
\end{proof}

\begin{proof}[Proof of \cref{thm:kron-residual-threshold}]
Let
\[
  \delta_\star=d_{\AI}(H,\Kron_{m,n})=d_{\AI}(H,P).
\]
Since \(G\in\Kron_{m,n}\), \(\delta_\star\le\rho\). By
\cref{cor:kron-residual-certificate},
\[
  0
  \le
  \frac12\rho^2-\frac12\delta_\star^2
  \le
  \frac12\eta^2,
\]
and hence \(\delta_\star\ge\sqrt{\max\{\rho^2-\eta^2,0\}}=\delta_-\). This
proves the distance bracket.

The family \(\Kron_{m,n}\) is closed under positive rescaling, so
\cref{thm:scale-closed-threshold} gives
\[
  K_{\Kron}^\star(H)
  \ge
  \exp(\alpha_{mn}\delta_\star)
  \ge
  \exp(\alpha_{mn}\delta_-).
\]
The upper bound \(K_{\Kron}^\star(H)\le\kappa_G\) holds because \(G\) is an
admissible Kronecker metric. If \(\kappa_G\le K\), then
\cref{prop:kron-relative-candidate} applied with \(G_c=G\) gives the displayed
finite \(D_{K,\Kron}\) upper bound.

It remains to prove the exact-projection condition bracket. Again by
\cref{cor:kron-residual-certificate},
\[
  d_{\AI}(G,P)\le\eta.
\]
Therefore every log-eigenvalue of \(G^{-1/2}PG^{-1/2}\) has absolute value at
most \(\eta\), and
\[
  e^{-\eta}G\preceq P\preceq e^\eta G.
\]
For any nonzero \(x\),
\[
  e^{-\eta}\frac{x^\top Hx}{x^\top Gx}
  \le
  \frac{x^\top Hx}{x^\top Px}
  \le
  e^\eta\frac{x^\top Hx}{x^\top Gx}.
\]
Taking suprema and infima over generalized Rayleigh quotients gives
\[
  \lambda_{\max}(P^{-1}H)\le e^\eta\lambda_{\max}(G^{-1}H),\qquad
  \lambda_{\min}(P^{-1}H)\ge e^{-\eta}\lambda_{\min}(G^{-1}H),
\]
so \(\kappa_P\le e^{2\eta}\kappa_G\). Interchanging \(G\) and \(P\) gives
\(\kappa_G\le e^{2\eta}\kappa_P\), equivalently
\(\kappa_P\ge e^{-2\eta}\kappa_G\).

Since \(P\) is an admissible Kronecker metric,
\(K_{\Kron}^\star(H)\le\kappa_P\).  The scale-closed lower bound and the
distance bracket give
\[
  \log K_{\Kron}^\star(H)
  \ge\alpha_{mn}\delta_\star
  \ge\alpha_{mn}\delta_-.
\]
Combining this with \(\log\kappa_P\le\log\kappa_G+2\eta\) proves the stated
projection-suboptimality chain.  The projection feasibility and
projection-infeasibility statements follow from the two-sided condition
bracket. The global impossibility certificate is the contrapositive of the
threshold lower bound.
\end{proof}

\begin{proof}[Proof of \cref{prop:kron-fixed-basis-relative}]
For
\[
  G=(V\diag(e^{b_j})V^\top)\otimes(U\diag(e^{a_i})U^\top),
\]
the matrices \(H\) and \(G\) commute in the basis \(V\otimes U\), and the
eigenvalues of \(G^{-1}H\) are
\[
  \exp(\ell_{ij}-a_i-b_j).
\]
Therefore
\[
  \log\kappa(G^{-1}H)
  =
  \max_{i,j}(\ell_{ij}-a_i-b_j)
  -
  \min_{i,j}(\ell_{ij}-a_i-b_j).
\]
The condition \(\kappa(G^{-1}H)\le K\) is thus equivalent to the existence of
a scalar \(c\) such that all residual log-eigenvalues lie in
\([c,c+\log K]\), which is exactly the displayed linear feasibility problem.
Minimizing the residual width over \(a,b\) gives the fixed-basis threshold
formula. Since \(\Kron_{U,V}\subset\Kron_{m,n}\), feasibility in the subfamily
is a full-family primal certificate; subfamily infeasibility gives no
full-family obstruction.
\end{proof}

\begin{proof}[Proof of \cref{thm:kron-fixed-basis-dual}]
For any \(Z\in\R^{m\times n}\),
\[
  \max_{ij}Z_{ij}-\min_{ij}Z_{ij}
  =
  \max_{p,q\in\Delta_{mn}}\ip{p-q}{Z},
\]
where \(\Delta_{mn}\) is the probability simplex over entries. The set
\(\{p-q:p,q\in\Delta_{mn}\}\) is exactly
\[
  \mathcal C=\{W:\sum_{ij}W_{ij}=0,\ \sum_{ij}|W_{ij}|\le2\}.
\]
Let
\[
  \mathcal A=\{A\in\R^{m\times n}:A_{ij}=a_i+b_j\}.
\]
Then
\[
  \tau^\star_{U,V}(H)=\min_{A\in\mathcal A}\max_{W\in\mathcal C}
  \ip{W}{L-A}.
\]
This is a finite-dimensional linear program. Its dual is
\[
  \max_{W\in\mathcal C\cap\mathcal A^\perp}\ip{W}{L},
\]
with no duality gap. The orthogonality condition is computed from
\[
  \ip{W}{A}
  =
  \sum_i a_i\sum_j W_{ij}
  +
  \sum_j b_j\sum_i W_{ij}.
\]
Thus \(W\in\mathcal A^\perp\) exactly when every row sum and every column sum
vanishes. These constraints imply \(\sum_{ij}W_{ij}=0\), so the dual feasible
set is the one displayed in the theorem.

If a feasible \(W\) has \(\ip{W}{L}>\log K\), then
\(\tau^\star_{U,V}(H)>\log K\), so the fixed-basis log-width target is
infeasible by \cref{prop:kron-fixed-basis-relative}. For any nonzero
zero-total \(W\), the identity
\(\sum W^+=\sum W^-=\frac12\norm{W}_1\) gives the normalized witness form when
\(\norm{W}_1=2\).
\end{proof}

\begin{proof}[Proof of \cref{thm:kron-certify-soundness}]
\Cref{prop:kron-fixed-basis-relative} proves soundness of the fixed-basis
feasible status.  \Cref{thm:kron-fixed-basis-dual} proves the obstruction
status.  The four full noncommuting statuses follow from
\cref{thm:kron-residual-threshold}. If none of those conditions holds, the
procedure records only \(\mathrm{INCONCLUSIVE}\), which makes no claim.
\end{proof}

\begin{proof}[Proof of \cref{cor:kron-certify-interval}]
The enclosure assumptions imply
\[
  \delta_-
  =
  \sqrt{\max\{\rho^2-\eta^2,0\}}
  \ge
  \sqrt{\max\{\underline\rho^2-\overline\eta^2,0\}}
  =\underline\delta.
\]
Therefore
\(K_{\Kron}^\star(H)\ge\exp(\alpha_{mn}\underline\delta)\), proving the
first implication.  The second follows from
\(\kappa_G\le\overline\kappa\).  For the exact projection, the bracket in
\cref{thm:kron-residual-threshold} and the enclosures give
\[
  e^{-2\overline\eta}\underline\kappa
  \le
  e^{-2\eta}\kappa_G
  \le
  \kappa_P
  \le
  e^{2\eta}\kappa_G
  \le
  e^{2\overline\eta}\overline\kappa.
\]
The remaining two implications follow immediately.  Returning
\(\mathrm{INCONCLUSIVE}\) when no strict certified comparison holds makes no
claim and is therefore sound.
\end{proof}

\begin{proof}[Proof of \cref{prop:kron-relative-candidate}]
The matrix
\[
  S_c^{-1}=H^{1/2}G_c^{-1}H^{1/2}
\]
has the same eigenvalues as \(G_c^{-1}H\), so
\(\kappa(S_c)=\kappa(G_c^{-1}H)\). Hence the displayed condition is exactly
membership of \(S_c\) in \(\Ck\), and \(S_c\in\mathcal C_{K,\Kron}(H)\).

By \cref{thm:kron-ai-projection}, \(\Kron_{m,n}\) is totally geodesic in the
ambient affine-invariant SPD cone. Therefore the affine-invariant geodesic
from \(G_0\) to \(G_c\) remains in \(\Kron_{m,n}\) and has length
\(d_{\AI}(G_0,G_c)\). Congruence by \(H^{-1/2}\) is an isometry for the
affine-invariant metric, so the relative path
\[
  S_t=H^{-1/2}G_tH^{-1/2}
\]
is an admissible path in \(\mathcal S_{\Kron}(H)\) from \(S_0\) to the target
point \(S_c\) with the same length. Taking the infimum over all admissible
paths gives
\[
  D_{K,\Kron}(S_0;H)\le d_{\AI}(G_0,G_c).
\]
The closed-form expression for the right-hand side is
\cref{prop:kron-product-distance}. The last statement is only the logical
one-way nature of a sufficient certificate.
\end{proof}

\begin{proof}[Proof of \cref{prop:kron-kkt}]
The objective gradient at \(Q\), in whitened coordinates and restricted to
\(\mathcal L_{m,n}\), is \(-\Pi_{\mathcal L}R_Q(M)\). For the path
\[
  Q(t)=Q^{1/2}\exp(tZ)Q^{1/2},
\]
the directional derivative of \(\log\lambda_{\max}\) is
\[
  \max_{\substack{u\in E_{\max}\\ \norm{u}=1}}u^\top Zu,
\]
and the directional derivative of \(\log\lambda_{\min}\) is
\[
  \min_{\substack{v\in E_{\min}\\ \norm{v}=1}}v^\top Zv.
\]
Thus the directional derivative of \(\omega=\log\lambda_{\max}-\log\lambda_{\min}\)
is the support function of the stated subdifferential. Since \(K>1\) has the
strict feasible point \(I_{mn}\), the convex KKT condition on the tangent
Hadamard submanifold gives
\[
  0\in-\Pi_{\mathcal L}R_\star+\mu\Pi_{\mathcal L}\partial\omega(Q_\star),
  \qquad \mu\ge0,
\]
with complementarity. This is the displayed equation; inactive constraints
have \(\mu=0\).
\end{proof}

\begin{proof}[Proof of \cref{prop:kfac-certificates}]
Projection onto the closed geodesically convex family \(\Kron_{m,n}\) satisfies
the Hadamard Pythagorean inequality. Applying it to
\[
  H_t,\qquad P_t=P_{\Kron}(H_t),\qquad G_t\in\Kron_{m,n}
\]
gives \(E_t^2\ge M_t^2+A_t^2\). Applying
\cref{thm:kron-k-target} with \(M=H_t\) and \(G_\star=P_t\) gives the two-sided
self-conditioned \(K\)-target bound. The computability statements are exactly
\cref{thm:kron-ai-projection,thm:kron-k-target,prop:kron-product-distance}.
\end{proof}

\begin{proof}[Proof of \cref{prop:proxy-inflation}]
From
\[
  (1-\varepsilon)H\preceq \widehat H\preceq (1+\varepsilon)H
\]
we obtain
\[
  \frac{1}{1+\varepsilon}
  \frac{v^\top\widehat H v}{v^\top Gv}
  \le
  \frac{v^\top H v}{v^\top Gv}
  \le
  \frac{1}{1-\varepsilon}
  \frac{v^\top\widehat H v}{v^\top Gv}.
\]
Taking minima and maxima over nonzero \(v\) gives the bound on generalized
condition numbers.
\end{proof}

\begin{proof}[Proof of \cref{prop:stoch-proxy-upper}]
On the stated Loewner event, \cref{prop:proxy-inflation} gives
\[
  \kappa(G^{-1}H)
  \le
  \frac{1+\varepsilon}{1-\varepsilon}
  \kappa(G^{-1}\widehat H)
  \le
  \frac{1+\varepsilon}{1-\varepsilon}\widehat K
  =K.
\]
The event has probability at least \(1-\delta\), which proves the claim.
\end{proof}

\section{Extensions}
\label{app:extensions}

\subsection{Low-rank spectral monotonicity}

Let
\[
  \mathcal Y_{K,r}(y)
  =
  \{z\in\Yk:\exists c\in\R,\ |\{i:z_i\ne y_i+c\}|\le r\}.
\]
Then
\[
  D_{K,r}^{\mathrm{spec}}(y)
  =
  \dist_{\ell^2}(y,\mathcal Y_{K,r}(y)).
\]
\begin{proof}[Proof of \cref{prop:lowrank-monotonicity}]
If \(r_1\le r_2\), then
\[
  \mathcal Y_{K,r_1}(y)\subseteq \mathcal Y_{K,r_2}(y)\subseteq \Yk,
\]
and therefore
\[
  D_{K,r_1}^{\mathrm{spec}}(y)
  \ge
  D_{K,r_2}^{\mathrm{spec}}(y)
  \ge
  D_K(y).
\]
When \(r\ge d\), the low-rank restriction disappears and equality with
\(D_K(y)\) holds.
\end{proof}

\end{document}